\documentclass{amsart}

\usepackage{amssymb}
\usepackage{mathrsfs}
\usepackage{stmaryrd}
\usepackage[colorlinks = true]{hyperref}
\usepackage{mathabx}

\usepackage{tikz}
\usetikzlibrary{cd}

\newcommand{\bk}{\Bbbk}
\newcommand{\Z}{\mathbb{Z}}
\newcommand{\Q}{\mathbb{Q}}
\newcommand{\F}{\mathbb{F}}

\newcommand{\bA}{\mathbb{A}}

\newcommand{\cC}{\mathcal{C}}

\newcommand{\scA}{\mathscr{A}}
\newcommand{\scB}{\mathscr{B}}
\newcommand{\scF}{\mathscr{F}}
\newcommand{\scG}{\mathscr{G}}
\newcommand{\scH}{\mathscr{H}}
\newcommand{\scL}{\mathscr{L}}

\newcommand{\Perv}{\mathsf{Perv}}
\newcommand{\pH}{{}^p\!\mathscr{H}}

\newcommand{\Spec}{\mathrm{Spec}}
\newcommand{\Dbc}{D^{\mathrm{b}}_{\mathrm{c}}}
\newcommand{\Db}{D^{\mathrm{b}}}

\newcommand{\ba}{\mathbf{a}}
\newcommand{\bb}{\mathbf{b}}
\newcommand{\bc}{\mathbf{c}}
\newcommand{\bi}{\mathbf{i}}
\newcommand{\bj}{\mathbf{j}}
\newcommand{\bh}{\mathbf{h}}

\newcommand{\sfC}{\mathsf{C}}
\newcommand{\sfZ}{\mathsf{Z}}

\newcommand{\scO}{\mathscr{O}}

\newcommand{\cE}{\mathcal{E}}
\newcommand{\cF}{\mathcal{F}}
\newcommand{\cL}{\mathcal{L}}
\newcommand{\cG}{\mathcal{G}}

\newcommand{\hGamma}{\widehat{\Gamma}}
\newcommand{\bGr}{\mathbf{Gr}}
\newcommand{\tbGr}{\widetilde{\mathbf{Gr}}}

\newcommand{\ctGr}{\widehat{\bGr}{}}
\newcommand{\simto}{\xrightarrow{\sim}}
\newcommand{\Gr}{\mathrm{Gr}}
\newcommand{\Fl}{\mathrm{Fl}}
\newcommand{\Alg}{\mathsf{Alg}}
\newcommand{\tsfC}{\widetilde{\sfC}}

\newcommand{\bt}{\mathbf{t}}
\newcommand{\nmin}{\overline{\text{\normalfont m\i n}}}

\makeatletter
\def\lotimes{\@ifnextchar_{\@lotimessub}{\@lotimesnosub}}
\def\@lotimessub_#1{\mathchoice{\mathbin{\mathop{\otimes}^L}_{#1}}%
  {\otimes^L_{#1}}{\otimes^L_{#1}}{\otimes^L_{#1}}}
\def\@lotimesnosub{\mathbin{\mathop{\otimes}^L}}
\makeatother

\newcommand{\tboxtimes}{\mathbin{\widetilde{\mathord{\boxtimes}}}}

\newcommand{\Hom}{\mathrm{Hom}}

\newcommand{\Gal}{\mathrm{Gal}}

\numberwithin{equation}{section}
\newtheorem{thm}{Theorem}[section]
\newtheorem{lem}[thm]{Lemma}
\newtheorem{prop}[thm]{Proposition}

\theoremstyle{definition}
\newtheorem{defn}[thm]{Definition}

\theoremstyle{remark}
\newtheorem{rmk}[thm]{Remark}
\newtheorem{ex}[thm]{Example}

\title{Higher nearby cycles and central sheaves on affine flag varieties}

\author{Pramod N. Achar}
\address{Department of Mathematics\\
  Louisiana State University\\
  Baton Rouge, LA 70803\\
  U.S.A}
\email{pramod.achar@math.lsu.edu}
\author{Simon Riche}
\address{Universit\'e Clermont Auvergne, CNRS, LMBP, F-63000 Clermont-Ferrand, France}
\email{simon.riche@uca.fr}
 
\thanks{P.A. was supported by NSF Grant Nos.~DMS-1802241 and DMS-2202012. This project has received funding from the European Research Council (ERC) under the European Union's Horizon 2020 research and innovation programme (S.R., grant agreement No. 101002592).}

\begin{document}

\begin{abstract}
In this paper we generalize and study a notion of (unipotent) nearby cycles over a higher dimensional base based on Be{\u\i}linson's description of unipotent nearby cycles, following an idea of Gaitsgory. This generalization, in the setting of affine Grassmannians, is required in recent work of Bezrukavnikov--Braverman--Finkelberg--Kazhdan~\cite{bbfk}.
\end{abstract}

\maketitle

\section{Introduction}

\subsection{}

Nearby cycles form a crucial ingredient in the local geometric Langlands program, in particular in Gaitsgory's construction of central (perverse) sheaves on the affine flag variety of a reductive algebraic group~\cite{gaitsgory}. In order to prove a certain technical property of this construction required in work of Bezrukavnikov~\cite{bez}, Gaitsgory has introduced in~\cite{gaitsgory-app} a somewhat ad-hoc construction of ``nearby cycles along a 2-dimensional base.'' 
In this paper we elaborate on this idea and explain how to define much more general ``nearby cycles'' functors over any finite-dimensional affine space. This generalization is used in recent work of Bezrukavnikov--Braverman--Finkelberg--Kazhdan on local L-factors, see~\cite{bbfk}.

\subsection{Be\texorpdfstring{{\u\i}}{i}linson's unipotent nearby cycles functor}

The starting point of this construction is Be{\u\i}linson's description of the unipotent nearby cycles functor using local systems on $\bA^1 \smallsetminus \{0\}$ associated with unipotent Jordan blocks. More specifically, fix a base field $\F$, and consider a scheme $X$ of finite type over $\bA^1 = \bA^1_\F$, with structure morphism $f : X \to \bA^1$. Fix also a field $\bk$ of coefficients (of characteristic invertible in $\F$), and let $\scF$ be a perverse sheaf on $f^{-1}(\bA^1 \smallsetminus \{0\})$. Then Be{\u\i}linson's construction provides, for any integer $a \geq 1$, a complex $\Psi_{f,a}(\scF)$ on $f^{-1}(0)$ concentrated in perverse degrees $-1$ and $0$, and morphisms $\Psi_{f,a}(\scF) \to \Psi_{f,b}(\scF)$ if $a \leq b$. The main observation is that these complexes ``stabilize'' to the unipotent nearby cycles $\Psi^{\mathrm{un}}_f(\scF)$ in the sense that:
\begin{itemize}
\item
if $a \gg 0$, for any $b \geq a$ the morphism $\pH^{-1}(\Psi_{f,a}(\scF)) \to \pH^{-1}(\Psi_{f,b}(\scF))$ is an isomorphism, and both objects identify with $\Psi^{\mathrm{un}}_f(\scF)$;
\item
for any $a$,
if $b \gg a$ the morphism $\pH^{0}(\Psi_{f,a}(\scF)) \to \pH^{0}(\Psi_{f,b}(\scF))$ vanishes.
\end{itemize}

\subsection{Higher dimensional version}

Gaitsgory's construction in~\cite{gaitsgory-app} is based on the same idea, but now for a scheme $X$ (of finite type) over $\bA^2$, and produces complexes indexed by a pair of nonnegative integers. In this generality, one cannot expect the same stabilization phenomenon as above for all perverse sheaves; but Gaitsgory simply defines his 2-dimensional nearby cycles functor using a homotopy colimit of these complexes. In general, this construction forces one to leave the constructible derived category; to avoid this problem, in our account of Gaitsgory's construction in~\cite{arbook} we proposed instead to restrict to complexes such that a stabilization property as above occurs. So, for us the $2$-dimensional nearby cycles functor is only ``partially defined.''
The starting point of this paper is the observation that this definition can be phrased so that it makes sense for a scheme $X$ (of finite type) over any affine space $\bA^d$. 

\subsection{Composition of higher nearby cycles functors}

The application of nearby cycles over a 2-dimensional base in~\cite{gaitsgory-app} uses a comparison with two related operations: (a) compute the nearby cycles of the given perverse sheaf successively along the 2 factors in $\bA^2$; (b) restrict the perverse sheaf to the preimage of the diagonal copy of $\bA^1$, and then compute the nearby cycles of (a shift of) this complex. Gaitsgory shows that his $2$-dimensional nearby cycles complex maps to each of the complexes obtained in (a) or (b), and that these maps are isomorphisms in the specific setting required in~\cite{bez}. In~\cite{bbfk} the authors consider (in a related setting) a variant of (a) where one considers a complex over $\bA^d$ and computes nearby cycles along each $\bA^1$-factor successively, and they assert that the result does not depend on the choice of order on the coordinates. To justify this assertion, here we introduce the (partially defined) $d$-dimensional nearby cycles functor, and show that, in the setting of~\cite{bbfk}, each of the complexes identifies canonically with the image of the inital complex under our functor.

To justify this fact, we interpret nearby cycles along each factor as the case $d=1$ of our construction, and show roughly that a composition of $d$- and $e$-dimensional nearby cycles functors receives a canonical map from a $(d+e)$-dimensional nearby cycles functor, which in the setting of~\cite{bbfk} is an isomorphism. In fact, in the body of the paper we consider a more general construction of higher nearby cycles and their compositions, which also covers the construction in (b) above, and thereby (in our opinion) clarifies the general picture. See Definition~\ref{def:main} for the main definition, and~\eqref{eqn:commnc-compose} for the construction involving the composition of such functors.

\subsection{Compatibilities}

In full generality, it is not reasonable to expect that the morphism in~\eqref{eqn:commnc-compose} is always an isomorphism. As in~\cite{gaitsgory-app}, what we show here is that the constructions of higher nearby cycles and of the morphisms considered above are compatible with smooth pullback and proper pushforward in an appropriate sense (see Lemmas~\ref{lem:compat-smooth}, \ref{lem:compat-proper}, \ref{lem:compat-compose-smooth} and \ref{lem:compat-compose-proper} for precise statements), and then study a product-type situation (see~\S\ref{ss:product}). This is sufficient to show that the $d$-dimensional nearby cycles are well defined in the setting of~\cite{bbfk}, and that their formation is compatible with composition, which amounts to the statement these authors require. See Theorem~\ref{thm:nc-main-new} and~\S\ref{ss:groupoid} for precise statements.

\subsection{Further comments}

There exists a general theory of nearby and vanishing cycles over general bases, which is much more elaborate than the version we consider here; see~\cite{illusie} for a brief account. In the companion paper~\cite{salmon2}, A.~Salmon explains the relationship between the latter theory and ours.

As explained above, the construction considered here is an extension of our study in~\cite[\S 9.4]{arbook}. There is one important difference in our treatment though: in~\cite{arbook} we introduced a general condition of ``iterated cleanness'' implying that the $2$-dimensional nearby cycles complex is well defined, and showed that this condition is satisfied in the setting at hand. It is not obvious (to us) how to extend this condition over a higher-dimensional base; here we bypass this question by proving a compatibility statement with proper pushforward which is stronger than its counterpart in~\cite{arbook} (compare Lemma~\ref{lem:compat-proper} and~\cite[Proposition~9.4.2]{arbook}) and allows us to reduce the question to the product-type setting.

In this paper we work in the setting of \'etale sheaves on schemes over fields, in order to meet the setting of~\cite{bbfk}. However all our statements have obvious variants for ``usual'' sheaves on complex algebraic varieties endowed with their analytic topology (and coefficients in a field), and all our proofs apply in both settings.

After we completed a preliminary version of this paper, it was pointed out to us that a special case of our construction had already been introduced by A. Salmon, with applications to the cohomology of shtukas; see~\cite{salmon} and Example~\ref{ex:Q-empty}.

\subsection{Acknowledgements}

We thank Roman Bezrukavnikov and Michael Finkelberg for suggesting that our study of Gaitsgory's construction could be used to prove the statement they use in~\cite{bbfk}, which was the starting point of this work, and for useful comments on a preliminary version. We thank the anonymous referee for detailed comments that helped improving the exposition of the paper. Finally, we thank Andrew Salmon for helpful discussions, and for agreeing to expand his comments into the paper~\cite{salmon2}.

\section{Preliminaries}

\subsection{Pointed maps and associated linear morphisms}
\label{ss:pointed-maps}

We fix a base field $\F$. Below, by a ``scheme'' we will mean an $\F$-scheme of finite type.
If $P$ is a finite set, we may consider the affine space
\[
\bA^P = \mathrm{Spec}(\F[X_p : p \in P])
\]
with coordinates indexed by $P$.  The \emph{generic part} of this affine space, denoted by $\bA^P_\eta$, is the open subscheme where no coordinates vanish:
\[
\bA^P_\eta = \{ (x_p)_{p \in P} \mid \text{$x_p \ne 0$ for all $p$} \}.
\]
(In case $P=\varnothing$, we interpret these definitions as $\bA^\varnothing = \bA^\varnothing_\eta = \mathrm{Spec}(\F)$.) 

For any finite set $P$ we
let $P_*$ denote the disjoint union $P \amalg \{*\}$, where $*$ denotes a new element. For finite sets $P$ and $Q$, a \emph{pointed map}
$\alpha: P_* \to Q_*$
is a function that satisfies $\alpha(*) = *$.  If $\alpha: P_* \to Q_*$ is a pointed map, there is an induced linear map of affine spaces
$\bar\alpha: \bA^Q \to \bA^P$
given by
\[
\bar\alpha((x_q)_{q \in Q}) = (y_p)_{p \in P} 
\qquad\text{where}\qquad
y_p = 
\begin{cases}
x_{\alpha(p)} & \text{if $\alpha(p) \in Q$,} \\
0 & \text{if $\alpha(p) = *$.}
\end{cases}
\]
We also set 
$\bA^P_{\eta,\alpha} = \{(x_p)_{p \in P} \in \bA^P \mid x_p \neq 0 \text{ if $\alpha(p) \neq *$}\}$; then the restriction of $\bar\alpha$ to $\bA^Q_\eta$ factors through a morphism $\bar\alpha_\eta : \bA^Q_\eta \to \bA^P_{\eta,\alpha}$.

\begin{rmk}
\label{rmk:bar-alpha}
 From the definition we see that $\bar\alpha$ is injective if $\alpha$ is surjective, and surjective if $\alpha$ is injective.  For arbitrary $\alpha$, one can decompose the situation into a combination of these two settings as follows: set $R = \alpha(P) \cap Q$. Then $\alpha$ decomposes in the obvious way as a composition of pointed maps
 \[
  P_* \xrightarrow{\alpha_1} R_* \xrightarrow{\alpha_2} Q_*,
 \]
and we have $\bar\alpha = \bar\alpha_1 \circ \bar\alpha_2$ where $\bar\alpha_1$ is injective and $\bar\alpha_2$ is surjective.
\end{rmk}

\subsection{Scheme morphisms associated with pointed maps}
\label{ss:schemes}

Now suppose we have a scheme $X$ equipped with a map
$f: X \to \bA^P$.
The \emph{generic part} of $X$ is defined by
$X_\eta = X \times_{\bA^P} \bA^P_\eta$;
the natural morphism $X_\eta \to \bA^P_\eta$ will be denoted $f_\eta$, and the embedding $X_\eta \to X$ will be denoted $\bj_X$ (or $\bj$ when no confusion is likely).

If $\alpha: P_* \to Q_*$ is a pointed map, then one can consider the schemes
\[
X^\alpha := X \times_{\bA^P} \bA^Q
\qquad\text{and}\qquad
X^\alpha_\eta := X \times_{\bA^P} \bA^Q_\eta
\]
where the fiber products are taken with respect to $\bar\alpha: \bA^Q \to \bA^P$ and its restriction to $\bA^Q_\eta$. The natural morphism $X^\alpha \to X$ will be denoted $\bi'_{X,\alpha}$ (or $\bi'_\alpha$), its restriction to $X^\alpha_\eta$ will be denoted $\bi_{X,\alpha}$ (or $\bi_\alpha$) and we will denote by
\[
f^\alpha : X^\alpha \to \bA^Q
\]
the natural projection morphism. By Remark~\ref{rmk:bar-alpha}, $\bi'_{\alpha}$ is a closed immersion if $\alpha$ is surjective, and is smooth and surjective if $\alpha$ is injective; moreover we have $\bi_\alpha = \bi'_\alpha \circ \bj_{X^\alpha}$. Note also that if $\alpha^{-1}(*)=\{*\}$ then $\bi_\alpha$ factors through a morphism
\[
 \bi_\alpha'' : X^\alpha_\eta \to X_\eta.
\]

\subsection{Jordan block local systems}
\label{ss:jordan-intro}

Let $\ell$ be a prime different from the characteristic of $\F$, and let $\bk$ be a topological field of one of the following two forms:
\begin{itemize}
\item an algebraic extension of $\Q_\ell$, equipped with the $\ell$-adic topology;
\item an algebraic extension of $\F_\ell$, equipped with the discrete topology.
\end{itemize}
Our goal in the rest of this section is to describe the construction of a family of indecompo\-sable $\bk$-local systems on $\bA^1 \smallsetminus \{0\}$ on which a generator of the geometric fundamental group acts by a unipotent Jordan block.

This construction is certainly well known when $\bk$ has characteristic $0$ (see, for instance,~\cite{beilinson,morel}), and it is relatively easy when $\F$ is separably closed (so that there is no Galois action to consider).  It is probably known in the generality we consider here, but as we could not find a suitable reference, we include the details. The construction uses identities involving some polynomials with rational coefficients, which are discussed in the next subsection.

\subsection{\texorpdfstring{$\Z$}{Z}-closed polynomials and binomial identities}
\label{ss:binomial}

Let us say that a polynomial $f(x) \in \Q[x]$ is \emph{$\Z$-closed} if it has the property that for all nonnegative integers $n \ge 0$, we have $f(n) \in \Z$.  Of course, any polynomial $f \in \Q[x]$ determines a continuous (for the $\ell$-adic topology) function $\Q_\ell \to \Q_\ell$.  Because the nonnegative integers $\Z_{\ge 0}$ are dense in $\Z_\ell$, if $f$ is $\Z$-closed, then it restricts to a function $\Z_\ell \to \Z_\ell$.  Furthermore, with $\bk$ as above, any $\Z$-closed polynomial $f \in \Q[x]$ determines a continuous function $f: \Z_\ell \to \bk$.

Any polynomial with coefficients in $\Z$ is obviously $\Z$-closed, but there are others: for instance, the polynomials $\delta_0, \delta_1, \ldots$ given by
\[
\delta_r(x) = \binom{x}{r} = \frac{x(x-1)\cdots (x-r+1)}{r!}.
\]
(When $r=0$, this expression is interpreted as $1$.)
For integers $r, s \ge 0$, these polynomials satisfy
\begin{equation}\label{eqn:binom-product}
\delta_r(x) \delta_s(x) = \sum_{i=0}^{\min \{r,s\}} \frac{(r+s-i)!}{(r-i)!(s-i)!i!} \delta_{r+s-i}(x).
\end{equation}
To prove this, observe that if we replace $x$ by a nonnegative integer $n$, the left-hand side is the coefficient of $u^rv^s$ in $(1+u)^n(1+v)^n$, which can also be written as
\[
(1 + u + v + uv)^n = \sum_{k=0}^n \binom{n}{k}(u+v+uv)^k = \sum_{k=0}^n \binom{n}{k} \sum_{\substack{g,h,i \ge 0 \\ g+h+i = k}} \frac{k!}{g!h!i!} u^{g+i} v^{h+i}.
\]
The coefficient of $u^rv^s$ is found by taking those terms where $g = r-i$, $h = s-i$, and $k = r+ s-i$.  This yields the right-hand side of~\eqref{eqn:binom-product}.

Next, consider the two-variable polynomial $\delta_r(tx) = \binom{tx}{r} \in \Q[x,t]$.  Since $\delta_0, \delta_1,  \ldots$ form a $\Q$-basis for $\Q[x]$, there is a unique expression of the form
\begin{equation}\label{eqn:delta-c}
\delta_r(tx) = \sum_{j=0}^r c_r^j(t) \delta_j(x),
\end{equation}
for some polynomials $c_r^j(t) \in \Q[t]$.  Comparing the coefficients of $x^r$ and evaluating at $x=0$, we see that
\begin{equation}\label{eqn:c-extreme}
c_r^r(t) = t^r \in \Z[t]
\qquad\text{and}\qquad
c_r^0(t) = \begin{cases}
1 & \text{if $r = 0$,} \\
0 & \text{if $r > 0$.}
\end{cases}
\end{equation}
In particular, these polynomials are $\Z$-closed.  

We claim that the $c_r^j$ are in fact $\Z$-closed for all $j$, i.e., that $c_r^j(m) \in \Z$ for integers $m \ge 0$.  We prove this by induction on $m$: if $r>0$ we obviously have $c_r^j(0) = 0$, and for $m > 0$, by Vandermonde's identity, we have
\[
\binom{mx}{r} = \sum_{g = 0}^r \binom{(m-1)x}{g} \binom{x}{r-g}
= \sum_{g=0}^r \sum_{h=0}^g c_g^h(m-1) \delta_h(x) \delta_{r-g}(x).
\]
Expanding the right-hand side using~\eqref{eqn:binom-product} and comparing coefficients of $\delta_j$, we find that
\[
c_r^j(m) = \sum_{\substack{0 \le h \le g \le r \\ 0 \le r-g+h-j \le \min\{h,r-g\}}} c_g^h(m-1) \frac{j!}{(j+g-r)!(j-h)!(r-g+h-j)!} \in \Z.
\]

\subsection{Functions on the Tate module}
\label{ss:functions-Zl}

Let $\overline{\F}$ be a separable closure of $\F$, and let $\Z_\ell(1)$ be the Tate module of $\overline{\F}{}^\times$, i.e., the inverse limit of the groups of $\ell^n$-th roots of unity in $\overline{\F}$. This is naturally a free $\Z_\ell$-module of rank $1$.  Let
\[
\cC(\Z_\ell(1),\bk) = \text{set of continuous functions $\Z_\ell(1) \to \bk$}.
\]
Here, ``continuous'' should be understood with respect to the profinite topology on $\Z_\ell(1)$ and the topology on $\bk$ indicated in~\S\ref{ss:jordan-intro}.  This set is a ring under pointwise multiplication.

The group $\Gal(\overline{\F}|\F)$ acts on $\Z_\ell(1)$ by the cyclotomic character $\chi_\ell: \Gal(\overline{\F}|\F) \to \Z_\ell^\times$.  There is an induced action of the semidirect product $\Gal(\overline{\F}|\F) \ltimes \Z_\ell(1)$ on $\cC(\Z_\ell(1), \bk)$, given by
\[
((\gamma, g) \cdot f)(h) = f(\chi_\ell(\gamma)^{-1}(h - g))
\]
for $\gamma \in \Gal(\overline{\F}|\F)$ and $g, h \in \Z_\ell(1)$.  This is an action by ring automorphisms.

To write down some explicit elements in $\cC(\Z_\ell(1), \bk)$, let us choose a generator $g$ for $\Z_\ell(1)$ as a $\Z_\ell$-module. Thus every element in $\Z_\ell(1)$ can be written uniquely as $ng$ for some $n \in \Z_\ell$.  Then for $r \geq 0$ there is a continuous function
\[
\varphi^g_r: \Z_\ell(1) \to \bk
\qquad\text{given by}\qquad
\varphi^g_r(ng) = (-1)^r \delta_r(n).
\]

\begin{lem}
\label{lem:jordan-basic}
Let $g \in \Z_\ell(1)$ be an element that generates $\Z_\ell(1)$ as a $\Z_\ell$-module, and for $a \geq 0$ let
\[
L_a = \mathrm{span}_\bk \{\varphi^g_0, \varphi^g_1, \ldots, \varphi^g_{a-1}\} \subset \cC(\Z_\ell(1), \bk).
\]
\begin{enumerate}
\item The subspace $L_a \subset \cC(\Z_\ell(1),\bk)$ is independent of the choice of $g$.\label{it:la-indep}
\item The subspace $L_a$ is stable under the action of $\Gal(\overline{\F}|\F) \ltimes \Z_\ell(1)$.\label{it:la-stab}
\item We have $\dim L_a = a$.\label{it:la-dim}
\item Any generator of $\Z_\ell(1)$ acts on $L_a$ by a unipotent Jordan block.\label{it:la-jordan}
\item For any integer $a \ge 1$, there is a canonical short exact of $\Gal(\overline{\F}|\F) \ltimes \Z_\ell(1)$-modules\label{it:la-coker}
\[
0 \to L_{a-1} \to L_a \to \bk(-a+1) \to 0.
\]
\item For any two integers $a, b \ge 1$, there is a canonical map of $\Gal(\overline{\F}|\F) \ltimes \Z_\ell(1)$-modules\label{it:la-mult}
$L_a \otimes_\bk L_b \to L_{a+b-1}$
whose image contains $L_{\max \{a,b\}}$.
\end{enumerate}
\end{lem}

\begin{proof}
\eqref{it:la-indep} If $g'$ is another generator for $\Z_\ell(1)$ as a $\Z_\ell$-module, then we have $g' = ug$ for some $u \in \Z_\ell^\times$.  It follows from~\eqref{eqn:delta-c} that
\[
\varphi_r^{g'} = \sum_{j=0}^{r} (-1)^{j+r} c_r^j(u) \varphi_j^g.
\]
This shows that the span of $\varphi_0^{g'}, \ldots, \varphi_{a-1}^{g'}$ is contained in the span of $\varphi_0^g, \ldots, \varphi_{a-1}^g$. By symmetry the opposite containment also holds, so $L_a$ is independent of the choice of $g$.

\eqref{it:la-stab}
We have
\[
((-g) \cdot \varphi_r^g)(ng) = \varphi_r^g((n+1)g) = (-1)^r \binom{n+1}{r} = (-1)^r\binom{n}{r-1} + (-1)^r\binom{n}{r}
\]
or, in other words,
\begin{equation}\label{eqn:gen-varphi}
(-g) \cdot \varphi_r^g = 
\begin{cases}
\varphi_0^g & \text{if $r = 0$,} \\
\varphi_r^g - \varphi_{r-1}^g & \text{if $r > 0$.}
\end{cases}
\end{equation}
This formula shows that $(-g)$ preserves $L_a$ and acts by a unipotent operator.  It follows that the cyclic subgroup $\Z g\subset \Z_\ell(1)$ preserves $L_a$ as well.  Since $\Z g$ is dense in $\Z_\ell(1)$, continuity considerations show that $L_a$ is preserved by all of $\Z_\ell(1)$.

Similarly, for $\gamma \in \Gal(\overline{\F}|\F)$, we have 
\begin{equation}\label{eqn:gal-varphi}
\gamma \cdot \varphi_r^{g} = \varphi_r^{\chi_\ell(\gamma)^{-1}g} = \sum_{j=0}^r (-1)^{j+r} c_r^j(\chi_\ell(\gamma)^{-1}) \varphi_j^g.
\end{equation}
Thus, $L_a$ is stable under $\Gal(\overline{\F}|\F)$, and hence under $\Gal(\overline{\F}|\F) \ltimes \Z_\ell(1)$.

\eqref{it:la-dim}
We wish to show that the functions $\varphi_0^g, \varphi_1^g, \ldots$ are linearly independent.  If not, find some nontrivial linear dependence relation
\[
b_0 \varphi_0^g + \cdots + b_m \varphi_m^g = 0
\]
with $b_0, \ldots, b_m \in \bk$.  We may assume that this relation is chosen so that $m$ as small as possible; then $b_m \ne 0$ and $m > 0$.  Apply the operator $f \mapsto f - (-g)\cdot f$ to this equation.  By~\eqref{eqn:gen-varphi}, we obtain
\[
b_1 \varphi_0^g + b_2\varphi_1^g + \cdots + b_m \varphi_{m-1}^g = 0,
\]
contradicting the minimality of $m$. We conclude that $\dim L_a = a$.  

\eqref{it:la-jordan}
This follows from~\eqref{eqn:gen-varphi} and the independence on $g$ (see~\eqref{it:la-indep}).

\eqref{it:la-coker}
By construction we have an embedding of representations $L_{a-1} \subset L_a$.
If $g$ is a generator of $\Z_\ell(1)$ as a $\Z_\ell$-module, then it is clear from~\eqref{eqn:gen-varphi} that $(-g)$ acts trivially on the quotient $L_a/L_{a-1}$.  Therefore, the group $\Z g \subset \Z_\ell(1)$ acts trivially, and by continuity, all of $\Z_\ell(1)$ acts trivially.

On the other hand, by~\eqref{eqn:gal-varphi}, an element $\gamma \in \Gal(\overline{\F}|\F)$ acts on $L_a/L_{a-1}$ by the scalar $c_{a-1}^{a-1}(\chi_\ell(\gamma)^{-1}) = \chi_\ell(\gamma)^{-a+1}$ (see~\eqref{eqn:c-extreme}). We conclude that $L_a/L_{a-1} \cong \bk(-a+1)$.

\eqref{it:la-mult}
Consider the multiplication map $\cC(\Z_\ell(1),\bk) \otimes_\bk \cC(\Z_\ell(1),\bk) \to \cC(\Z_\ell(1),\bk)$.  This map is compatible with the action of $\Gal(\overline{\F}|\F) \ltimes \Z_\ell(1)$.  It follows from~\eqref{eqn:binom-product} that the multiplication map restricts to the desired map
\[
L_a \otimes_\bk L_b \to L_{a+b-1}.
\]
The image of this map contains $\varphi_a = \varphi_a \varphi_0$ and $\varphi_b = \varphi_0 \varphi_b$.  By~\eqref{eqn:gen-varphi}, it contains all of $L_a$ and $L_b$, and thus it contains $L_{\max \{a,b\}}$.
\end{proof}

In the next lemma (which will not be used below) we discuss the case when $\bk$ has characteristic $0$; in particular, these results show that in this case our construction recovers the local systems considered in~\cite{beilinson,morel}.

\begin{lem}
\label{lem:jordan-semis}
Assume that $\bk$ has characteristic $0$.
\begin{enumerate}
\item 
\label{it:la-exp}
For $r \ge 0$, there is an isomorphism of $\Gal(\overline{\F}|\F)$-modules
\[
L_a \cong \bk \oplus \bk(-1) \oplus \cdots \oplus \bk(-a+1).
\]
In terms of the right-hand side, the action of $g \in \Z_\ell(1)$ is given by
\begin{equation}\label{eqn:la-exp}
g \mapsto \exp
\left[\begin{smallmatrix}
0 & \langle -g,-\rangle\\
& 0 & \langle -g,-\rangle \\
&& \ddots & \ddots \\
&&& 0 & \langle -g,-\rangle \\
&&&& 0
\end{smallmatrix}\right]
\end{equation}
where $\langle{-},{-}\rangle: \Z_\ell(1) \otimes_{\Z_\ell} \bk(-r) \to \bk(-r+1)$ is the natural pairing map.
\item
\label{it:la-ses}
For $a \le b$, there is a short exact sequence of $\Gal(\overline{\F}|\F) \ltimes \Z_\ell(1)$-modules
\[
0 \to L_a \to L_b \to L_{b-a}(-a) \to 0.
\]
\end{enumerate}
\end{lem}

\begin{proof}
By definition we have
\[
\bk(-r) = \bk \otimes_{\Z_\ell} \Hom_{\Z_\ell}(\Z_\ell(r),\Z_\ell) = \bk \otimes_{\Z_\ell} \Hom_{\Z_\ell}(\underbrace{\Z_\ell(1) \otimes \cdots \otimes \Z_\ell(1)}_{\text{$r$ copies}}, \Z_\ell).
\]
Define a $\bk$-linear map $\theta_r: \bk(-r) \to \cC(\Z_\ell(1),\bk)$ as follows: for $u \in \Hom_{\Z_\ell}(\Z_\ell(r),\Z_\ell)$, we let $\theta_r(1 \otimes u): \Z_\ell(1) \to \bk$ be the function given by
\[
\theta_r(1 \otimes u)(g) = \textstyle\frac{1}{r!} u(g \otimes \cdots \otimes g).
\]
Then $\theta_r$ is $\Gal(\overline{\F}|\F)$-equivariant and nonzero.  We claim that $\theta_r$ induces an isomorphism
\begin{equation}\label{eqn:twist-power}
\bk(-r) \cong \left\{ f \in \cC(\Z_\ell(1),\bk) \,\Big|\, 
\begin{array}{@{}c@{}}
\text{$f(ng) = n^r f(g)$ for all} \\ \text{$n \in \Z_\ell$ and all $g \in \Z_\ell(1)$}
\end{array} \right\}.
\end{equation}
Indeed, it is clear that the image of $\theta_r$ is contained in the right-hand side of~\eqref{eqn:twist-power}.  Moreover, the latter set is $1$-dimensional, because any element satisfying this condition is determined by its value on some generator of $\Z_\ell(1)$ as a $\Z_\ell$-module.

Let us describe the pairing $\langle{-},{-}\rangle: \Z_\ell(1) \otimes_{\Z_\ell} \bk(-r) \to \bk(-r+1)$ in terms of $\theta_r$ and $\theta_{r-1}$.  For $g \in \Z_\ell(1)$ and $f$ belonging to the right-hand side of~\eqref{eqn:twist-power}, we have
\begin{equation}\label{eqn:twist-pairing}
\langle g, f\rangle(g) = r f(g).
\end{equation}

For the remainder of this proof, we assume that $g$ is a generator of $\Z_\ell(1)$ as a $\Z_\ell$-module.  Consider the continuous function
\[
\nu_r^g: \Z_\ell(1) \to \bk
\qquad\text{given by}\qquad
\nu_r^g(ng) = n^r / r!.
\]
It is clear that $\nu_r^g$ spans the right-hand side of~\eqref{eqn:twist-power}. On the other hand, $\nu_0^g, \ldots, \nu_{a-1}^g$ is a basis for $L_a$, so $L_a \cong \bk \oplus \cdots \oplus \bk(-a+1)$ as a $\Gal(\overline{\F}|\F)$-module. By~\eqref{eqn:twist-pairing}, $\langle g, \nu_r^g \rangle(g) = 1/(r-1)!$, so
\[
\langle g, \nu_r^g\rangle = \nu_{r-1}^g.
\]

In the usual $\Z_\ell(1)$-action on $L_a$, the action of $g$ on $\nu_r^g$ is given by
\[
(g \cdot \nu^g_r)(ng) = \frac{(n-1)^r}{r!} 
=\sum_{j=0}^r \frac{(-1)^j}{j!} \nu^g_{r-j}(ng) 
= \sum_{j=0}^r \frac{1}{j!} \underbrace{\langle -g, \langle \cdots \langle -g,}_{\text{$j$ times}} \nu^g_r \rangle \cdots \rangle(ng).
\]
This shows that~\eqref{eqn:la-exp} holds when $g$ is a generator of $\Z_\ell(1)$ as a $\Z_\ell$-module.  It then holds in general by continuity.

\eqref{it:la-ses}
We continue to let $g$ be a fixed generator of $\Z_\ell(1)$ as a $\Z_\ell$-module.  Let $\partial^a: L_b \to L_{b-a}$ be the map given by
\[
\partial^a(\nu_r^g) = 
\begin{cases}
\nu_{r-a}^g & \text{if $a \le r \le b-1$,} \\
0 & \text{otherwise.}
\end{cases}
\]
If we identify $L_b$ with the space of polynomial functions $\Z_\ell \to \bk$ of degree${}\le b-1$, then $\partial^a$ is the $a$-fold differentiation operator.  It is easy to check that $\partial^a$ is $\Z_\ell(1)$-equivariant, but not $\Gal(\overline{\F}|\F)$-equivariant.  Specifically, for $\gamma \in \Gal(\overline{\F}|\F)$ and $f \in L_b$, we have
\[
\partial^a(\gamma \cdot f) = \chi_\ell(\gamma)^{-a} \gamma \cdot \partial^a(f).
\]
We conclude that $\partial^a$ induces an isomorphism $L_b/L_a \cong L_{b-a}(-a)$.
\end{proof}

\begin{rmk}
When $\bk$ has characteristic $\ell$, the isomorphism in Lemma~\ref{lem:jordan-semis}\eqref{it:la-exp} may be false. A direct calculation shows that
\[
\delta_2(tx) = \frac{tx(tx-1)}{2} = t^2 \delta_2(x) + \frac{t^2 -t}{2} \delta_1(x),
\]
so that $c_2^1(t) = \frac{1}{2}(t^2 -t)$.  Let $g$ be a generator of $\Z_\ell(1)$ as a $\Z_\ell$-module.  By~\eqref{eqn:gal-varphi}, in the basis $\{\varphi_0^g, \varphi_1^g, \varphi_2^g\}$ for $L_3$, an element $\gamma \in \Gal(\overline{\F}|\F)$ acts by
\[
\begin{bmatrix}
1 \\
& \chi_\ell(\gamma)^{-1} & \frac{1}{2}(\chi_\ell(\gamma)^{-1} - \chi_\ell(\gamma)^{-2}) \\
&& \chi_\ell(\gamma)^{-2}
\end{bmatrix}.
\]
Suppose now that $\ell = 2$, and that $\F = \F_p$ for some prime $p \equiv 3 \pmod 4$.  Take $\gamma \in \Gal(\overline{\F}|\F)$ to be the Frobenius map, so that $\chi_\ell(\gamma) = p$.  Then $\gamma$ acts on $L_3$ by
\[
\begin{bmatrix}
1 \\
& p^{-1} & p^{-3} \binom{p}{2} \\
&& p^{-2}
\end{bmatrix}
=
\begin{bmatrix}
1 \\ & 1 & 1 \\ && 1
\end{bmatrix}.
\]
In particular, the action of $\Gal(\overline{\F}|\F)$ on $L_3$ is not semisimple!
\end{rmk}

\subsection{Definition of the local systems}
\label{ss:locsys}

Consider the affine line $\bA^1$ over $\F$.  Recall that there is a canonical exact sequence
\[
\pi_1^{\mathrm{geom}}(\bA^1 \smallsetminus \{0\}, 1) \hookrightarrow \pi_1(\bA^1 \smallsetminus \{0\}, 1) \twoheadrightarrow \Gal(\overline{\F}|\F).
\]
This sequence is split by the choice of the unit $\F$-point in $\bA^1 \smallsetminus \{0\}$, and we identify $\pi_1(\bA^1 \smallsetminus \{0\},1) = \Gal(\overline{\F} |\F) \ltimes \pi_1^{\mathrm{geom}}(\bA^1 \smallsetminus \{0\},1)$.  We also have a canonical quotient map $\pi_1^{\mathrm{geom}}(\bA^1 \smallsetminus \{0\}, 1) \twoheadrightarrow \Z_\ell(1)$, which is compatible with the action of $\Gal(\overline{\F}|\F)$ on both groups.  Combining these observations, we obtain a canonical surjective map
\[
\pi_1(\bA^1 \smallsetminus \{0\},1) \twoheadrightarrow \Gal(\overline{\F} |\F) \ltimes \Z_\ell(1).
\]
Thus, any continuous $\bk$-representation of $\Gal(\overline{\F} |\F) \ltimes \Z_\ell(1)$ gives rise to an \'etale $\bk$-local system on $\bA^1$.  For $a \geq 0$, we let $\scL_a$ be the local system on $\bA^1 \smallsetminus \{0\}$ corresponding to the representation $L_a$ as in Lemma~\ref{lem:jordan-basic}.

The pullback of $\scL_a$ to $\bA^1_{\overline{\F}} \smallsetminus \{0\}$ is the local system corresponding to the continuous representation of $\pi_1^{\mathrm{geom}}(\bA^1 \smallsetminus \{0\}, 1)$ given by the restriction of $L_a$ to this subgroup. By construction this action factors through an action of the commutative group $\Z_\ell(1)$. It therefore admits a canonical action of $\Z_\ell(1)$ by automorphisms of local systems.

By Lemma~\ref{lem:jordan-basic}, for $a \geq 1$ there exists a canonical exact sequence
\begin{equation}
\label{eqn:surjection-trivial}
\scL_{a-1} \hookrightarrow \scL_a \twoheadrightarrow \scL_1(-a+1)
\end{equation}
where $\scL_1$ is the trivial local system. Iterating these embeddings we obtain, for any $a \leq b$, a canonical embedding
\begin{equation}
 \label{eqn:jordan-inc1}
 \scL_a \hookrightarrow \scL_b.
\end{equation}
For any $a,b \geq 1$, there is a canonical morphism
\[
\scL_a \otimes \scL_b \to \scL_{a+b-1}
\]
whose image contains $\scL_{\max(a,b)}$, and such that the following diagram commutes, where the vertical maps are the surjections appearing in~\eqref{eqn:surjection-trivial}:
\[
\begin{tikzcd}
\scL_{a} \otimes \scL_{b} \ar[r] \ar[d]
& \scL_{a + b- 1} \ar[d]
\\
\scL_1(-a - b + 2) \ar[r, "\mathrm{id}"] & \scL_1(-a - b + 2).
\end{tikzcd}
\]
Moreover, these morphisms are compatible with the embeddings~\eqref{eqn:jordan-inc1} 
in the obvious way. 
Iterating this construction one obtains, for any $k \geq 1$ and any integers $a_1, \ldots, a_k$, a morphism of local systems on $\bA^1 \smallsetminus \{0\}$
\begin{equation}
\label{eqn:jordan-mult}
\scL_{a_1} \otimes \cdots \otimes \scL_{a_k} \to \scL_{a_1+\cdots+a_k - k+1}.
\end{equation}

\section{Higher nearby cycles}
\label{sec:higher-nc}

\subsection{Setting}
\label{ss:setting}

Let $X$, $f$, $\alpha$ be as in~\S\ref{ss:schemes}.
The goal of this section is to explain the construction, for $\scF$ a perverse sheaf on $X_\eta$, of a perverse sheaf
\[
\Upsilon_f^\alpha(\scF) \in \Perv(X^\alpha_\eta, \bk),
\]
together with a collection parametrized by $\alpha^{-1}(*) \cap P$ of pairwise commuting actions of $\Z_\ell(1)$ on its pullback to $\overline{\F}$. 
We warn the reader that this construction is ``partial'': it will be defined only for certain perverse sheaves $\scF$. We do not have any general criterion which guarantees that this construction works, but we do have tools (see Lemmas~\ref{lem:compat-smooth}, \ref{lem:compat-proper} and \ref{lem:product-upexist}) that can be used to show that this is the case in certain settings, where it gives rise to very important objects (see Section~\ref{sec:app}).

In the case where $|P|=1$ and $Q=\varnothing$, our construction will reduce to Be{\u\i}linson's description of the unipotent nearby cycles functor, and when $|P|=2$ and $Q=\varnothing$ it corresponds to the theory of ``nearby cycles along a 2-dimensional base'' developed in~\cite{gaitsgory-app} and studied in~\cite[\S9.4]{arbook}; see Example~\ref{ex:Q-empty}.

\subsection{Definition}

Let $\ba: P \to \Z_{\ge 1}$ be a function.  Define a local system $\scL_\ba$ on $\bA^P_\eta$ by
\[
\scL_\ba = \bigboxtimes_{p \in P} \scL_{\ba(p)}.
\]
For any $p \in P$, we have an action of $\Z_\ell(1)$ on the pullback of $\scL_{\ba}$ to $\overline{\F}$ induced by the action on the factor labelled by $p$.
If $\ba,\bb: P \to \Z_{\ge 1}$ are two functions, we say that $\ba \le \bb$ if $\ba(p) \le \bb(p)$ for all $p \in P$.  If $\ba \le \bb$, then~\eqref{eqn:jordan-inc1} gives rise to an embedding of local systems
\begin{equation}
\label{eqn:jordan-inc2}
\scL_\ba \hookrightarrow \scL_\bb
\end{equation}
on $\bA^P_\eta$.  In the special case where $\ba(p) = \bb(p)$ for all but one element $p_0$ of $P$, and moreover $\bb(p_0)=\ba(p_0)+1$, by~\eqref{eqn:surjection-trivial} the cokernel is a Tate twist of a local system of the same form: specifically, we have a short exact sequence
\begin{equation}
\label{eqn:locsys-ses}
\scL_\ba \hookrightarrow \scL_\bb \twoheadrightarrow \scL_\bc(\ba(p_0))
\quad\text{where}\quad
\begin{cases}
\bc(p) = \ba(p) = \bb(p) & \text{for all $p \ne p_0$,} \\
\bc(p_0)=1.
\end{cases}
\end{equation}

Let us say that $\ba: P \to \Z_{\ge 1}$ is \emph{$\alpha$-special} (with respect to a given pointed map $\alpha: P_* \to Q_*$) if for each $p \in \alpha^{-1}(Q)$ we have $\ba(p) = 1$.

\begin{defn}
\label{def:main}
Let $f: X \to \bA^P$ be a morphism of schemes. Let $\scF \in \Perv(X_\eta,\bk)$, and let $\alpha: P_* \to Q_*$ be a pointed map.  If $\ba, \bb: P \to \Z_{\ge 1}$ satisfy $\ba \le \bb$, then for any $i \in \Z$ there is a natural map
\begin{equation}
\label{eqn:unc-jordan}
\pH^i(\bi_{\alpha}^* \bj_*(\scF \otimes f_\eta^*\scL_\ba)) \to \pH^i(\bi_{\alpha}^* \bj_*(\scF \otimes f_\eta^*\scL_\bb))
\end{equation}
induced by~\eqref{eqn:jordan-inc2}.
We say that the \emph{$\alpha$-nearby cycles of $\scF$ are well defined} if 
\begin{itemize}
\item
for $i = |Q| - |P|$, there exists $N \in \Z_{\geq 0}$ such that if $\ba$ is $\alpha$-special and satisfies $\ba(p) \geq N$ for any $p \in \alpha^{-1}(*) \cap P$, then for any $\bb \geq \ba$ $\alpha$-special the map~\eqref{eqn:unc-jordan} is an isomorphism;
\item
for $i \ne |Q| - |P|$, for any $\ba$ $\alpha$-special there exists $\bb \geq \ba$ $\alpha$-special such that the map~\eqref{eqn:unc-jordan} vanishes.
\end{itemize}
If these conditions are satisfied, we set
\[
\Upsilon^\alpha_f(\scF) = \varinjlim_{\substack{\ba: P \to \Z_{\ge 1} \\ \text{$\alpha$-special}}} \pH^{|Q| - |P|}(\bi_{\alpha}^* \bj_*(\scF \otimes f_\eta^*\scL_\ba)).
\]
In this case,
for any $p \in \alpha^{-1}(*) \cap P$, after pullback to $\overline{\F}$ we have an action of $\Z_\ell(1)$ on $\Upsilon^\alpha_f(\scF)$ induced by the corresponding action on $\scL_\ba$, called the monodomy action associated with $p$. These actions pairwise commute.
\end{defn}

\begin{rmk}
\phantomsection
\label{rmk:def-Upsilon}
\begin{enumerate}
 \item 
 It should be clear that, although we omit the adjective ``unipotent'' from our terminology for simplicity, what we consider in Definition~\ref{def:main} is an extension of the construction of the \emph{unipotent part} of the nearby cycles functor.
 \item
 \label{it:ind-schemes}
 In the definition above one can allow more generally $X$ to be an ind-scheme over $\F$ which is of ind-finite type. By definition a complex on such an ind-scheme is supported on a closed subscheme, and the functor of Definition~\ref{def:main} can be calculated on a complex by restriction to a closed subscheme supporting the given complex. All the statements of this section hold in this generality, replacing the condition ``proper'' by ``ind-proper,'' and the property ``smooth'' by ``representable by a smooth morphism.''
 \end{enumerate}
\end{rmk}

We will sometimes write $\Upsilon^\alpha_X(\scF)$ instead of $\Upsilon^\alpha_f(\scF)$. This construction is obviously functorial in the sense that if $\scF,\scG \in \Perv(X_\eta,\bk)$ are such that the $\alpha$-nearby cycles of $\scF$ and $\scG$ are well defined and if $u : \scF \to \scG$ is a morphism, then we have a natural morphism $\Upsilon^\alpha_f(u) : \Upsilon^\alpha_f(\scF) \to \Upsilon^\alpha_f(\scG)$ which intertwines the monodromy actions (after pullback to $\overline{\F}$).

\begin{lem}
The partially defined functor $\Upsilon^\alpha_f$ is exact
in the sense that if
\[
\scF_1 \hookrightarrow \scF_2 \twoheadrightarrow \scF_3
\]
is a short exact sequence in $\Perv(X_\eta,\bk)$ such that the $\alpha$-nearby cycles of $\scF_1$, $\scF_2$ and $\scF_3$ are well defined, then the induced morphisms
\[
\Upsilon^\alpha_f(\scF_1) \to \Upsilon^\alpha_f(\scF_2) \to \Upsilon^\alpha_f(\scF_3)
\]
form a short exact sequence in $\Perv(X^\alpha_\eta,\bk)$.
\end{lem}

\begin{proof}
For any $\ba : P \to \Z_{\geq 1}$ we have an exact sequence
\begin{multline*}
\pH^{|Q|-|P|-1}(\bi_{\alpha}^* \bj_*(\scF_3 \otimes f_\eta^*\scL_\ba)) \to \pH^{|Q|-|P|}(\bi_{\alpha}^* \bj_*(\scF_1 \otimes f_\eta^*\scL_\ba)) \\
\to \pH^{|Q|-|P|}(\bi_{\alpha}^* \bj_*(\scF_2 \otimes f_\eta^*\scL_\ba))
\to \pH^{|Q|-|P|}(\bi_{\alpha}^* \bj_*(\scF_3 \otimes f_\eta^*\scL_\ba)) \\
\to \pH^{|Q|-|P|+1}(\bi_{\alpha}^* \bj_*(\scF_1 \otimes f_\eta^*\scL_\ba)).
\end{multline*}
Fix $N$ such that the map~\eqref{eqn:unc-jordan} is an isomorphism for $i=|Q|-|P|$, for any $\ba$ which is $\alpha$-special with $\ba(p) \geq N$ for $p \in \alpha^{-1}(*) \cap P$ and any $\bb$ which is $\alpha$-special and satisfies $\bb \geq \ba$, for the three complexes $\scF_1$, $\scF_2$ and $\scF_3$. Fix $\ba$ which is $\alpha$-special and satisfies $\ba(p) \geq N$ for $p \in \alpha^{-1}(*) \cap P$. Fix $\bb \geq \ba$ which is $\alpha$-special and such that the morphism
\[
\pH^{|Q|-|P|+1}(\bi_{\alpha}^* \bj_*(\scF_1 \otimes f_\eta^*\scL_\ba)) \to \pH^{|Q|-|P|+1}(\bi_{\alpha}^* \bj_*(\scF_1 \otimes f_\eta^*\scL_\bb))
\]
vanishes, and then fix $\bc \geq \bb$ which is $\alpha$-special and such that the morphism
\[
\pH^{|Q|-|P|-1}(\bi_{\alpha}^* \bj_*(\scF_3 \otimes f_\eta^*\scL_\bb)) \to \pH^{|Q|-|P|-1}(\bi_{\alpha}^* \bj_*(\scF_3 \otimes f_\eta^*\scL_\bc))
\]
vanishes. Considering the commutative diagram
\[
\begin{tikzcd}
\pH^{|Q|-|P|-1}(\bi_{\alpha}^* \bj_*(\scF_3 \otimes f_\eta^*\scL_\bb)) \ar[r] \ar[d] & \pH^{|Q|-|P|}(\bi_{\alpha}^* \bj_*(\scF_1 \otimes f_\eta^*\scL_\bb)) \ar[d] \\
\pH^{|Q|-|P|-1}(\bi_{\alpha}^* \bj_*(\scF_3 \otimes f_\eta^*\scL_\bc)) \ar[r] & \pH^{|Q|-|P|}(\bi_{\alpha}^* \bj_*(\scF_1 \otimes f_\eta^*\scL_\bc))
\end{tikzcd}
\]
in which all maps are the natural ones, using the fact that the right vertical map is an isomorphism and the left one vanishes, we obtain that the upper horizontal map vanishes. Similarly, considering the commutative diagram
\[
\begin{tikzcd}
\pH^{|Q|-|P|}(\bi_{\alpha}^* \bj_*(\scF_3 \otimes f_\eta^*\scL_\ba)) \ar[r] \ar[d] &
\pH^{|Q|-|P|+1}(\bi_{\alpha}^* \bj_*(\scF_1 \otimes f_\eta^*\scL_\ba)) \ar[d] \\
\pH^{|Q|-|P|}(\bi_{\alpha}^* \bj_*(\scF_3 \otimes f_\eta^*\scL_\bb)) \ar[r]
& \pH^{|Q|-|P|+1}(\bi_{\alpha}^* \bj_*(\scF_1 \otimes f_\eta^*\scL_\bb))
\end{tikzcd}
\]
where the left vertical map is an isomorphism and the right one vanishes, we obtain that the lower horizontal map vanishes. This means that in the exact sequence considered at the beginning of the proof (for the map $\bb$) the first and last morphisms vanish. Since the second, third and fourth terms in this sequence identify with $\Upsilon^\alpha_f(\scF_1)$, $\Upsilon^\alpha_f(\scF_2)$ and $\Upsilon^\alpha_f(\scF_3)$ respectively, this finishes the proof.
\end{proof}

\begin{rmk}
Recall the notation of~\S\ref{ss:pointed-maps}, and set
$X_{\eta,\alpha} = X \times_{\bA^P} \bA^P_{\eta,\alpha}$.
Then the immersion $\bj_X$ factors as a composition
\[
X_\eta \xrightarrow{\bj_{X,\alpha,1}} X_{\eta,\alpha} \xrightarrow{\bj_{X,\alpha,2}} X,
\]
and $\bi_{X,\alpha}$ factors through a morphism
$\bh_{X,\alpha} : X^\alpha_\eta \to X_{\eta,\alpha}$.
Hence for any function $\ba$ we have an identification
\begin{equation}
\label{eqn:simplification-Upsilon}
\bi_{\alpha}^* \bj_*(\scF \otimes f_\eta^*\scL_\ba) = \bh_{X,\alpha}^* (\bj_{X,\alpha,1})_*(\scF \otimes f_\eta^*\scL_\ba).
\end{equation}
\end{rmk}

\begin{rmk}
\label{rmk:decomposition-alpha}
Consider the decomposition $\alpha=\alpha_2 \circ \alpha_1$ from Remark~\ref{rmk:bar-alpha}. Then a map $\ba : P \to \Z_{\geq 1}$ is $\alpha$-special iff it is $\alpha_1$-special, and we have $\bi'_{X,\alpha}=\bi'_{X,\alpha_1} \circ \bi'_{X^{\alpha_1},\alpha_2}$, so for any $\alpha$-special $\ba$ we have
 \[
  \bi_{X,\alpha}^* (\bj_X)_*(\scF \otimes f_\eta^*\scL_\ba) \cong 
  (\bi_{X^{\alpha_1},\alpha_2})^* (\bi'_{X,\alpha_1})^* (\bj_X)_*(\scF \otimes f_\eta^*\scL_\ba).
 \]
 Since $(\alpha_2)^{-1}(*)=\{*\}$ we have a morphism $\bi''_{X^{\alpha_1},\alpha_2}$, which is easily seen to be smooth and surjective, and an identification
 \[
  \bi_{X,\alpha}^* (\bj_X)_*(\scF \otimes f_\eta^*\scL_\ba) \cong
  (\bi''_{X^{\alpha_1},\alpha_2})^* (\bi_{X,\alpha_1})^* (\bj_X)_*(\scF \otimes f_\eta^*\scL_\ba).
 \]
Since $\bi''_{X^{\alpha_1},\alpha_2}$ is smooth of relative dimension $|Q|-|R|$ the functor $(\bi''_{X^{\alpha_1},\alpha_2})^*[|Q|-|R|]$ is exact with respect to the perverse t-structure (see~\cite[\S 4.2.4]{bbd} or~\cite[Proposition~3.6.1]{achar-book}), so for any $i \in \Z$ we have
\begin{multline}
\label{eqn:isom-composition-alpha}
 \pH^i(\bi_{X,\alpha}^* (\bj_X)_*(\scF \otimes f_\eta^*\scL_\ba)) \cong \\
 (\bi''_{X^{\alpha_1},\alpha_2})^* [|Q|-|R|] \left( \pH^{i-|Q|+|R|}(\bi_{X,\alpha_1}^* (\bj_X)_*(\scF \otimes f_\eta^*\scL_\ba)) \right).
\end{multline}
Since $\bi''_{X^{\alpha_1},\alpha_2}$ is also surjective, $(\bi''_{X^{\alpha_1},\alpha_2})^*[|Q|-|R|]$ is also faithful on perverse sheaves (see~\cite[Theorem~3.6.6]{achar-book}),\footnote{In fact $\bi''_{X^{\alpha_1},\alpha_2}$ has geometrically connected fibers, so that $(\bi''_{X^{\alpha_1},\alpha_2})^*[|Q|-|R|]$ is even fully faithful on perverse sheaves, see~\cite[Proposition~4.2.5]{bbd} or~\cite[Theorem~3.6.6]{achar-book}.}
hence it detects isomorphisms.
We deduce that the $\alpha$-nearby cycles of $\scF$ are well defined iff so are the $\alpha_1$-nearby cycles of $\scF$, and that in this case we have
\[
 \Upsilon^\alpha_f(\scF) = (\bi''_{X^{\alpha_1},\alpha_2})^*\Upsilon^{\alpha_1}_f(\scF)[|Q|-|R|].
\]
\end{rmk}

\subsection{First properties}
\label{ss:first-properties}

The following lemma makes more precise the perverse degrees one has to consider in Definition~\ref{def:main}.

\begin{lem}
\label{lem:degrees}
 Let $\scF \in \Perv(X_\eta,\bk)$, and let $\alpha: P_* \to Q_*$ be a pointed map. For any map $\ba : P \to \Z_{\geq 1}$, we have
$\pH^i(\bi_{\alpha}^* \bj_*(\scF \otimes f_\eta^*\scL_\ba))=0$
unless $|Q|-|P| \leq i \leq |Q|-|\mathrm{im}(\alpha) \cap Q|$.
\end{lem}

\begin{proof}
Remark~\ref{rmk:decomposition-alpha} reduces the proof to the case $\alpha$ is surjective, which we assume from now on.
 Set $r:=|P|-|Q|$. Then $\bi'_\alpha$ is a closed immersion; more specifically, it can be written as a composition $i_r \circ \cdots \circ i_1$ where each $i_j$ is a closed immersion whose complementary open immersion is affine. (In fact it suffices to remark this for $\bar\alpha$, where the decomposition is obtained by writing this map as a composition of embeddings of codimension-$1$ linear subspaces.) 
Then the claim follows from the fact that if $i: X \to Y$ is a closed immersion with affine complement, 
the functor $i^*$ sends any complex concentrated in perverse degrees between $a$ and $b$ to a complex concentrated in perverse degrees between $a-1$ and $b$, see~\cite[Corollaire~4.1.10]{bbd}.
\end{proof}

We deduce the following property.

\begin{lem}
\label{lem:ph-injective}
Let $\scF \in \Perv(X_\eta,\bk)$, and let $\alpha: P_* \to Q_*$ be a pointed map. For any two functions $\ba,\bb:P \to \Z_{\ge 1}$ with $\ba \le \bb$, the natural map
\[
\pH^{|Q|-|P|}(\bi_{\alpha}^* \bj_*(\scF \otimes f_\eta^*\scL_\ba)) \to
\pH^{|Q|-|P|}(\bi_{\alpha}^* \bj_*(\scF \otimes f_\eta^*\scL_\bb))
\]
is injective.
\end{lem}
\begin{proof}
By induction, we can reduce to the case where $\ba(p) = \bb(p)$ for all but one element of $P$, say $p_0$, and moreover $\bb(p_0)=\ba(p_0)+1$.  In this case, define $\bc$ as in~\eqref{eqn:locsys-ses}.  That short exact sequence gives rise to a distinguished triangle
\[
\bi_{\alpha}^* \bj_*(\scF \otimes f_\eta^*\scL_\ba) \to
\bi_{\alpha}^* \bj_*(\scF \otimes f_\eta^*\scL_\bb) \to
\bi_{\alpha}^* \bj_*(\scF \otimes f_\eta^*\scL_\bc)(\ba(p_0)) \xrightarrow{[1]}.
\]
Lemma~\ref{lem:degrees} applies to all three terms, and then the present lemma follows by examining the long exact sequence in perverse cohomology.
\end{proof}

\subsection{Reformulation}

In the following lemma we show that Definition~\ref{def:main} can be formulated in a slightly different way.

\begin{lem}
\label{lem:altdefn}
Let $\scF \in \Perv(X_\eta,\bk)$, and let $\alpha: P_* \to Q_*$ be a pointed map. The $\alpha$-nearby cycles of $\scF$ are well defined if and only if the following conditions hold.
\begin{itemize}
\item There exists $N \in \Z_{\geq 0}$ such that if $\ba$ is $\alpha$-special and satisfies $\ba(p) \geq N$ for any $p \in \alpha^{-1}(*) \cap P$, then for any $\bb \geq \ba$ $\alpha$-special the natural map
\[
\pH^{|Q|-|P|}(\bi_{\alpha}^* \bj_*(\scF \otimes f_\eta^*\scL_\ba)) \to \pH^{|Q|-|P|}(\bi_{\alpha}^* \bj_*(\scF \otimes f_\eta^*\scL_\bb))
\]
is an isomorphism.
\item For any $\ba$ $\alpha$-special there exists $\bb \geq \ba$ $\alpha$-special such that the natural map
\[
{}^p \tau^{> |Q| - |P|} (\bi_{\alpha}^* \bj_*(\scF \otimes f_\eta^*\scL_\ba)) \to {}^p \tau^{>|Q|-|P|}(\bi_{\alpha}^* \bj_*(\scF \otimes f_\eta^*\scL_\bb))
\]
vanishes.
\end{itemize}
\end{lem}

\begin{proof}
By Lemma~\ref{lem:degrees}, the map~\eqref{eqn:unc-jordan} automatically vanishes for $i < |Q| - |P|$.  In view of this, it is clear that the conditions in the present lemma imply those in Definition~\ref{def:main}.

Conversely, assume that the conditions in Definition~\ref{def:main} hold.  Let $\ba: P \to \Z_{\ge 1}$ be an $\alpha$-special function, and define a sequence of functions $\ba_1, \ba_2, \ldots$ inductively as follows: set $\ba_1 = \ba$, and if $\ba_1, \ldots, \ba_{n-1}$ are already defined, choose $\ba_n \ge \ba_{n-1}$ such that
\[
 \pH^{i}(\bi_{\alpha}^* \bj_*(\scF \otimes f_\eta^*\scL_{\ba_{n-1}})) \to \pH^{i}(\bi_{\alpha}^* \bj_*(\scF \otimes f_\eta^*\scL_{\ba_{n}}))
\]
vanishes for $i > |Q| - |P|$.  (By Lemma~\ref{lem:degrees} again, there are only finitely many degrees $i > |Q| - |P|$ in which the objects above are nonzero, so finding such a $\ba_n$ requires only finitely many invocations of Definition~\ref{def:main}.)  Set
\[
M_j = {}^p \tau^{>|Q|-|P|} (\bi_{\alpha}^* \bj_*(\scF \otimes f_\eta^*\scL_{\ba_j})),
\qquad j = 1, 2,\ldots.
\]
By Lemma~\ref{lem:vancrit} below, there is an integer $N \geq 1$ such that $M_1 \to M_N$ is the zero map. The second condition of the lemma is then satisfied by $\bb = \ba_N$.
\end{proof}

\begin{lem}
\label{lem:vancrit}
Let $\mathscr{T}$ be a triangulated category equipped with a non-degenerate t-structure, and suppose we have a sequence of objects and maps
\[
M_1 \xrightarrow{\phi_1} M_2 \xrightarrow{\phi_2} M_3 \to \cdots
\]
such that the following conditions hold:
\begin{enumerate}
\item there exist integers $a \le b$ such that for all $j$, the t-cohomology ${}^t \hspace{-1.5pt} \scH^i(M_j)$ vanishes unless $a \le i \le b$;
\item for any $j \geq 1$ and $i \in \Z$, the map ${}^t \hspace{-1.5pt} \scH^i(\phi_j)$
vanishes.
\end{enumerate}
Then the composition
$\phi_{b-a+1} \circ \phi_{b-a} \circ \cdots \circ \phi_1: M_1 \to M_{b-a+2}$ vanishes.
\end{lem}

 \begin{proof}
 We proceed by induction on $b - a$.  If $b-a = 0$, then $M_j = {}^t \hspace{-1.5pt} \scH^a(M_j)[-a]$ for all $j$, and the maps $\phi_j$ are already all zero by assumption, so the claim is clear.
 
 Otherwise, for any $j$ we set $M'_j = {}^t \tau^{\ge a+1}M_j$, and let $\phi'_j: M'_j \to M'_{j+1}$ be the induced map.  By induction, the map $\phi'_{b-a} \circ \cdots \circ \phi'_1: M'_1 \to M'_{b-a+1}$ vanishes.  Let $\psi = \phi_{b-a} \circ \cdots \circ \phi_1: M_1 \to M_{b-a+1}$, and consider the commutative diagram
 \[
 \begin{tikzcd}
 {}^t \hspace{-1.5pt} \scH^a(M_1)[-a] \ar[d, "0"'] \ar[r] & M_1 \ar[d, "\psi"'] \ar[r] & M'_1 \ar[d, "0"] \ar[r, "+1"] & {} \\
 {}^t \hspace{-1.5pt} \scH^a(M_{b-a+1})[-a] \ar[d, "0"'] \ar[r] & M_{b-a+1} \ar[d, "\phi_{b-a+1}"] \ar[r] & M'_{b-a+1} \ar[d, "\phi'_{b-a+1}"] \ar[r, "+1"] & {} \\
 {}^t \hspace{-1.5pt} \scH^a(M_{b-a+2})[-a] \ar[r] & M_{b-a+2} \ar[r] & M'_{b-a+2} \ar[r, "+1"] & {}
 \end{tikzcd}
 \]
 where all rows are distinguished triangles. An examination of the long exact sequence of $\Hom$-groups shows that $\psi$ must be induced by a map
 \[
 \theta:  M_1 \to {}^t \hspace{-1.5pt} \scH^a(M_{b-a+1})[-a].
 \]
 The composition $\phi_{b-a+1} \circ \psi$ is equal to the composition
 \[
 M_1 
 \xrightarrow{\theta} {}^t \hspace{-1.5pt} \scH^a(M_{b-a+1})[-a] \to M_{b-a+1} \xrightarrow{\phi_{b-a+1}} M_{b-a+2}.
 \]
 But the composition of the last two arrows is $0$, so $\phi_{b-a+1} \circ \psi = 0$, which concludes the proof.
 \end{proof}

\subsection{Examples}

To illustrate Definition~\ref{def:main} we next consider some special cases.

\begin{ex}
 Assume that $\alpha^{-1}(*)=\{*\}$. Then there exists only one $\alpha$-special map, namely the constant map with value $1$, and the corresponding local system is constant (of rank $1$). In this case, we interpret the conditions in Definition~\ref{def:main} as requiring that $\pH^i(\bi_{\alpha}^* \bj_*(\scF))=0$ if $i \neq |Q|-|P|$. Note that we have $\bi_\alpha = \bj \circ \bi_\alpha''$, hence
$\bi_{\alpha}^* \bj_*(\scF) = (\bi_\alpha'')^* \scF$.
It follows that the $\alpha$-nearby cycles of $\scF$ are well defined if and only if $(\bi_\alpha'')^*\scF[|Q|-|P|]$ is perverse, and that if this is the case we have
\[
\Upsilon^\alpha_f(\scF) \cong (\bi_\alpha'')^*\scF[|Q|-|P|].
\]
\end{ex}

\begin{ex}
\label{ex:Q-empty}
Assume now that $Q=\varnothing$. In this case, there exists a unique map $\alpha : P_* \to Q_*=\{*\}$ (which will therefore be omitted from the notation), and any map $\ba : P \to \Z_{\geq 1}$ is special. If $n=|P|$, we will speak of \emph{$n$-dimensional nearby cycles} instead of $\alpha$-nearby cycles in this case. More specifically:
\begin{enumerate}
\item
In case $n=1$, the constructions above amount to those of Be{\u\i}linson~\cite{beilinson} in his description of the unipotent nearby cycles functor and its monodromy action, see~\cite{morel} (see also~\cite{reich} or~\cite[\S 9.2]{arbook} for the analogous construction in the complex analytic setting); in particular, the $1$-dimensional nearby cycles of $\scF$ are well defined for any $\scF$, and compute the unipotent part of the nearby cycles $\Psi_f(\scF)$. The monodromy action on the pullback to $\overline{\F}$ is the inverse of the standard monodromy action.
\item
In case $n=2$, the considerations above specialize to the setting studied in~\cite[\S 9.4]{arbook} (following an idea of Gaitsgory in~\cite{gaitsgory-app}); in particular,~\cite[Proposition~9.4.7]{arbook} gives a sufficient condition under which the $2$-dimensional nearby cycles of $\scF$ are well defined and can be computed in terms of iterated unipotent nearby cycles.
\end{enumerate}
This case is also considered (for general $n$) in~\cite{salmon}, where appropriate versions of Lemmas~\ref{lem:compat-smooth} and~\ref{lem:compat-proper} below are also obtained.
\end{ex}

\subsection{Compatibilities}

\begin{lem}
\label{lem:classical-nc}
If $\alpha: P_* \to Q_*$ is surjective, $|Q| = |P| - 1$, and $|\alpha^{-1}(*)| = 2$, then the $\alpha$-nearby cycles of $\scF$ are well defined.
\end{lem}

\begin{proof}
Our assumptions imply that there is exactly one element $p \in P$ with $\alpha(p) = *$. It is clear that the datum of an $\alpha$-special map $\ba : P \to \Z_{\geq 1}$ is equivalent to the datum of a nonnegative integer (corresponding to $\ba(p)$). If $\pi_p: \bA^P \to \bA^1$ is the projection onto the $p$th coordinate, then the construction of the $\alpha$-nearby cycles of $\scF$ with respect to $f$ amounts to the construction of the $1$-dimensional nearby cycles of $\scF$ with respect to $\pi_p \circ f$, which are well defined by Example~\ref{ex:Q-empty}.
\end{proof}

For the next statements we fix a pointed map $\alpha : P_* \to Q_*$.
Given a morphism $g : Y \to X$, we will denote by
$g_\eta: Y_\eta \to X_\eta$ and
$g^\alpha_\eta: Y^\alpha_\eta \to X^\alpha_\eta$
the morphisms obtained by base change.

\begin{lem}
\label{lem:compat-smooth}
Let $g: Y \to X$ be a smooth morphism of relative dimension $d$, and let $\scF \in \Perv(X_\eta,\bk)$.
\begin{enumerate}
\item 
If the $\alpha$-nearby cycles of $\scF$ are well defined, then so are the $\alpha$-nearby cycles of $g_\eta^*\scF[d]$.
\item If $g$ is surjective, and if the $\alpha$-nearby cycles of $g_\eta^*\scF[d]$ are well defined, then so are the $\alpha$-nearby cycles of $\scF$.
\end{enumerate}
In either case, there is a natural isomorphism
\[
\Upsilon^\alpha_{f g}(g_\eta^*\scF[d]) \cong (g^{\alpha}_\eta)^* \Upsilon^\alpha_f(\scF)[d].
\]
\end{lem}

\begin{proof}
 The first claim follows from the smooth base change theorem and t-exactness of shifted smooth pullbacks. The second claim follows from the fact that pullback under a smooth surjective morphism is faithful on perverse sheaves and detects isomorphisms, as in Remark~\ref{rmk:decomposition-alpha}. Details are left to the reader.
\end{proof}

\begin{lem}
\label{lem:compat-proper}
Let $g: Y \to X$ be a proper morphism, and let $\scF \in \Perv(Y_\eta,\bk)$.  Assume that the following conditions hold:
\begin{enumerate}
\item the $\alpha$-nearby cycles of $\scF$ are well defined;
\item both $(g_{\eta})_* \scF$ and $(g^\alpha_{\eta})_* \Upsilon^\alpha_{f g}(\scF)$ are perverse.
\end{enumerate}
Then the $\alpha$-nearby cycles of $(g_{\eta})_*\scF$ are well defined, and there is a natural isomorphism
\[
\Upsilon^\alpha_f((g_{\eta})_*\scF) \cong (g^\alpha_{\eta})_* \Upsilon^\alpha_{f g}(\scF).
\]
\end{lem}

\begin{proof}
To simplify notation we set $r=|Q|-|P|$ and $h=f g$. For any $\alpha$-special functions $\ba, \bb: P \to \Z_{\ge 1}$ with $\ba \le \bb$, we can form the following commutative diagram, in which the columns are truncation distinguished triangles (in the top row, we use Lemma~\ref{lem:degrees} to identify ${}^p\tau^{\le r}(-)$ with $\pH^{r}(-)[-r]$):
\begin{equation}\label{eqn:cptprop-calc0}
\hbox{\small$
\begin{tikzcd}[column sep=small, row sep=small]
 \pH^{r}(\bi_{Y,\alpha}^*\bj_{Y*}(\scF \otimes h_\eta^*\scL_\ba))[-r] \ar[r] \ar[d] &
 \pH^{r}(\bi_{Y,\alpha}^*\bj_{Y*}\scF \otimes h_\eta^*\scL_\bb)[-r] \ar[d] \\
\bi_{Y,\alpha}^*\bj_{Y*}(\scF \otimes h_\eta^*\scL_\ba) \ar[r] \ar[d] &
   \bi_{Y,\alpha}^*\bj_{Y*}(\scF \otimes h_\eta^*\scL_\bb) \ar[d] \\
{}^p\tau^{>r}\bi_{Y,\alpha}^*\bj_{Y*}(\scF \otimes h_\eta^*\scL_\ba) \ar[r] \ar[d, "+1"] & 
 {}^p\tau^{>r}\bi_{Y,\alpha}^*\bj_{Y*}(\scF \otimes h_\eta^*\scL_\bb) \ar[d, "+1"] \\
{} & {}
\end{tikzcd}
$}
\end{equation}
Since the $\alpha$-nearby cycles of $\scF$ are well defined, by Lemma~\ref{lem:altdefn}, we may choose $\ba$ such that the top horizontal map is an isomorphism for any $\bb \geq \ba$ (so that these objects identify with $\Upsilon^\alpha_{h}(\scF)$), and then choose $\bb$ such that the bottom horizontal map is $0$.  

By base change and the projection formula, we have
\[
g^\alpha_{\eta*}\bi_{Y,\alpha}^*\bj_{Y*}(\scF \otimes h_\eta^*\scL_\ba)
\cong \bi_{X,\alpha}^*\bj_{X*}g_{\eta*}(\scF \otimes g_\eta^*f_\eta^*\scL_\ba)
\cong \bi_{X,\alpha}^* \bj_{X*}((g_{\eta*}\scF) \otimes f_\eta^*\scL_\ba).
\]
Thus, applying $g^\alpha_{\eta*}$ to~\eqref{eqn:cptprop-calc0}, we obtain a diagram
\begin{equation}
\label{eqn:cptprop-calc}
\hbox{\small$
\begin{tikzcd}[column sep=small, row sep=small]
 g^\alpha_{\eta*}\pH^{r}(\bi_{Y,\alpha}^*\bj_{Y*}(\scF \otimes h_\eta^*\scL_\ba))[-r] \ar[r, "\sim"] \ar[d] &
  g^\alpha_{\eta*}\pH^{r}(\bi_{Y,\alpha}^*\bj_{Y*}(\scF \otimes h_\eta^*\scL_\bb))[-r] \ar[d] \\
\bi_{X,\alpha}^*\bj_{X*}(g_{\eta*}\scF \otimes f_\eta^*\scL_\ba) \ar[r] \ar[d] &
   \bi_{X,\alpha}^*\bj_{X*}(g_{\eta*}\scF \otimes f_\eta^*\scL_\bb) \ar[d] \\
 g^\alpha_{\eta*}({}^p\tau^{>r}\bi_{Y,\alpha}^*\bj_{Y*}(\scF \otimes h_\eta^*\scL_\ba)) \ar[r, "0"] \ar[d, "+1"] & 
  g^\alpha_{\eta*}({}^p\tau^{>r}\bi_{Y,\alpha}^*\bj_{Y*}(\scF \otimes h_\eta^*\scL_\bb)) \ar[d, "+1"] \\
{} & {}
\end{tikzcd}$}
\end{equation}
whose columns are distinguished triangles.
The objects in the top row are identified with $g^\alpha_{\eta*}\Upsilon^\alpha_{h}(\scF)[-r]$; in particular, by assumption, they are concentrated in perverse degree $r$.  Since $g_{\eta*}\scF$ is assumed to be perverse, Lemma~\ref{lem:degrees} tells us that the objects in the middle row live in perverse degrees${}\ge r$. It follows that
\[
\pH^i(g^\alpha_{\eta*}({}^p\tau^{>r}\bi_{Y,\alpha}^*\bj_{Y*}(\scF \otimes h_\eta^*\scL_\ba))) = 0
\qquad \text{for $i \le r-2$,}
\]
and likewise for $\scL_\bb$.  Taking perverse cohomology, we obtain the following commutative diagram with exact columns:
\[
\hbox{\small$\begin{tikzcd}[column sep=small, row sep=small]
 \pH^{r-1}(g^\alpha_{\eta*}({}^p\tau^{>r}\bi_{Y,\alpha}^*\bj_{Y*}(\scF \otimes h_\eta^*\scL_\ba))) \ar[r, "0"] \ar[d, hook] & 
 \pH^{r-1}(g^\alpha_{\eta*}({}^p\tau^{>r}\bi_{Y,\alpha}^*\bj_{Y*}(\scF \otimes h_\eta^*\scL_\bb))) \ar[d, hook] \\
 g^\alpha_{\eta*}\pH^{r}(\bi_{Y,\alpha}^*\bj_{Y*}(\scF \otimes h_\eta^*\scL_\ba)) \ar[r, "\sim"] \ar[d] &
  g^\alpha_{\eta*}\pH^{r}(\bi_{Y,\alpha}^*\bj_{Y*}(\scF \otimes h_\eta^*\scL_\bb)) \ar[d] \\
\pH^{r}(\bi_{X,\alpha}^*\bj_{X*}(g_{\eta*}\scF \otimes f_\eta^*\scL_\ba))\ar[r] \ar[d, two heads] &
   \pH^{r}(\bi_{X,\alpha}^*\bj_{X*}(g_{\eta*}\scF \otimes f_\eta^*\scL_\bb)) \ar[d, two heads] \\
 \pH^{r}(g^\alpha_{\eta*}({}^p\tau^{>r}\bi_{Y,\alpha}^*\bj_{Y*}(\scF \otimes h_\eta^*\scL_\ba))) \ar[r, "0"] &
 \pH^{r}(g^\alpha_{\eta*}({}^p\tau^{>r}\bi_{Y,\alpha}^*\bj_{Y*}(\scF \otimes h_\eta^*\scL_\bb))).
\end{tikzcd}$}
\]
Here, the third horizontal arrow is injective by Lemma~\ref{lem:ph-injective}. Since the composition of the topmost horizontal arrow with the topmost right vertical arrow is injective, the latter morphism must be injective, which implies that 
the topmost term in the left-hand column vanishes. By one of the four-lemmas, the $0$ morphism on the fourth line is injective, so that the bottommost term in this column also vanishes. We deduce that we actually have
\[
\pH^i(g^\alpha_{\eta*}({}^p\tau^{>r}\bi_{Y,\alpha}^*\bj_{Y*}(\scF \otimes h_\eta^*\scL_\ba))) = 0
\qquad \text{for $i \leq r$.}
\]
The same reasoning also applies to $\bb$, which implies that the columns of~\eqref{eqn:cptprop-calc} can be identified with truncation distinguished triangles: that whole diagram can be rewritten as as
\begin{equation}
\label{eqn:cptprop-calc2}
\hbox{\small$\begin{tikzcd}[column sep=small, row sep=small]
 \pH^{r}(\bi_{X,\alpha}^*\bj_{X*}(g_{\eta*}\scF \otimes f_\eta^*\scL_\ba))[-r] \ar[r, "\sim"] \ar[d] &
 \pH^{r}(\bi_{X,\alpha}^*\bj_{X*}(g_{\eta*}\scF \otimes f_\eta^*\scL_\bb))[-r] \ar[d] \\
\bi_{X,\alpha}^*\bj_{X*}(g_{\eta*}\scF \otimes f_\eta^*\scL_\ba) \ar[r] \ar[d] &
   \bi_{X,\alpha}^*\bj_{X*}(g_{\eta*}\scF \otimes f_\eta^*\scL_\bb) \ar[d] \\
{}^p\tau^{>r}\bi_{X,\alpha}^*\bj_{X*}(g_{\eta*}\scF \otimes f_\eta^*\scL_\ba) \ar[r, "0"] \ar[d, "+1"] & 
 {}^p\tau^{>r}\bi_{X,\alpha}^*\bj_{X*}(g_{\eta*}\scF \otimes f_\eta^*\scL_\bb) \ar[d, "+1"] \\
{} & {}
\end{tikzcd}$}
\end{equation}

Our argument shows that the top (resp.~bottom) row of~\eqref{eqn:cptprop-calc2} is an isomorphism (resp.~zero) whenever the corresponding row of~\eqref{eqn:cptprop-calc0} has the same property.  By 
Lemma~\ref{lem:altdefn}, we conclude that the $\alpha$-nearby cycles of $g_{\eta*}\scF$ are well defined.  The identification of~\eqref{eqn:cptprop-calc} with~\eqref{eqn:cptprop-calc2} shows that $\Upsilon^\alpha_f(g_{\eta*}\scF) \cong g^\alpha_{\eta*} \Upsilon^\alpha_{f g}(\scF)$.
\end{proof}

\subsection{Compositions of higher nearby cycles: construction}

Let $\alpha: P_* \to Q_*$ and $\beta: Q_* \to R_*$ be pointed maps, and let $\scF \in \Perv(X_\eta,\bk)$ be an object which satisfies the following properties:
\begin{itemize}
\item the $\alpha$-nearby cycles and the $\beta\alpha$-nearby cycles of $\scF$ are well defined;
\item the $\beta$-nearby cycles of $\Upsilon_f^\alpha(\scF)$ are well defined.
\end{itemize}
In the rest of this subsection we explain how, in this setting, one can define a canonical morphism
\begin{equation}
\label{eqn:commnc-compose}
\Upsilon^{\beta \alpha}_f(\scF) \to \Upsilon^\beta_{f^\alpha}(\Upsilon_f^\alpha(\scF)).
\end{equation}

Let $\bc: P \to \Z_{\ge 1}$ be a $\beta\alpha$-special function.  Recall that this means that $\bc(p) \ne 1$ implies $\beta(\alpha(p)) = *$.  Define two new functions $\ba,\bb: P \to \Z_{\ge 1}$ by
\[
\ba(p) = 
\begin{cases}
\bc(p) & \text{if $\alpha(p) = *$,} \\
1 & \text{otherwise,}
\end{cases}
\qquad
\bb(p) =
\begin{cases} 
\bc(p) & \text{if $\beta(\alpha(p)) = *$ but $\alpha(p) \ne *$,} \\
1 & \text{otherwise.}
\end{cases}
\]
We clearly have that $\ba$ is $\alpha$-special, and that
$\scL_\bc \cong \scL_\ba \otimes \scL_\bb$.
Next, define $\bb': Q \to \Z_{\ge 1}$ by
\[
\bb'(q) = 
- |\alpha^{-1}(q)| + 1 + \sum_{p \in \alpha^{-1}(q)} \bb(p).
\]
We claim that $\bb'$ is $\beta$-special.  Indeed, if $\beta(q) \ne *$, then the summation involves only elements $p$ satisfying $\beta(\alpha(p)) \ne *$, and the claim follows from the fact that $\bc$ is $\beta\alpha$-special. Note also that if $\beta(q)=*$ we have
\begin{equation}
\label{eqn:b'-c}
\bb'(q) = 
- |\alpha^{-1}(q)| + 1 + \sum_{p \in \alpha^{-1}(q)} \bc(p).
\end{equation}

Recall the open subscheme $\bA^P_{\eta,\alpha}$ from~\S\ref{ss:pointed-maps} and the morphism $\bar\alpha_{\eta} : \bA^Q_{\eta} \to \bA^P_{\eta,\alpha}$. Note that the local system $\scL_{\bb}$ on $\bA^P_\eta$ extends (uniquely) to a local system $\scL_{\bb,\alpha}$ on $\bA^P_{\eta,\alpha}$. 
We claim that there exists a natural morphism 
\begin{equation}\label{eqn:jordan-pullback}
\bar\alpha_\eta^*\scL_{\bb,\alpha} \to \scL_{\bb'}
\end{equation}
of local systems on $\bA^Q_\eta$.  Indeed, for $q \in Q$, the $q$th copy of $\bA^1$ in $\bA^Q$ is mapped under $\bar\alpha$ to the diagonal copy of $\bA^1$ inside $\bA^{\alpha^{-1}(q)}$.
It follows that 
\[
\bar\alpha_\eta^*\scL_{\bb,\alpha} \cong \bigboxtimes_{q \in Q} \Big( \bigotimes_{p \in \alpha^{-1}(q)} \scL_{\bb(p)} \Big).
\]
The morphisms~\eqref{eqn:jordan-mult} provide a map
\[
\bigotimes_{p \in \alpha^{-1}(q)} \scL_{\bb(p)} \to \scL_{\bb'(q)}
\]
for each $q$; taking the external tensor product over $q$, we obtain~\eqref{eqn:jordan-pullback}.

Note now that the restriction of $\bar\alpha$ to $\bA^Q_{\eta,\beta}$ factors through $\bA^P_{\eta,\beta\alpha}$, which allows to define the morphism
$\bi_{X,\alpha,\beta} : X^\alpha_{\eta,\beta} \to X_{\eta,\beta\alpha}$
by base change.
We consider the commutative diagram as follows, where the unlabelled arrow is the obvious open immersion:
\[
\begin{tikzcd}[row sep=tiny]
X^{\beta\alpha}_\eta \ar[ddd, "\bj_{X^{\beta\alpha}}"'] \ar[rd, "\bh_{X^\alpha,\beta}" near start] && 
  X^\alpha_\eta \ar[ddd, "\bj_{X^\alpha}" near start] \ar[rrd, "\bh_{X,\alpha}"] \ar[ld, "\bj_{X^\alpha,\beta,1}"'] &&&
  X_\eta \ar[ddd, "\bj_X"] \ar[ld, "\bj_{X,\alpha,1}" near end] \ar[lldd, "\bj_{X,\beta\alpha,1}"' near start, bend right, crossing over] \\
  & X^{\alpha}_{\eta,\beta} \ar[rrd, "\bi_{X,\alpha,\beta}" near end, crossing over] \ar[rdd, "\bj_{X^\alpha,\beta,2}" description] &&& X_{\eta,\alpha} \ar[rdd, "\bj_{X,\alpha,2}" description] \ar[ld] & \\
  &&& X_{\eta,\beta\alpha} \ar[rrd, "\bj_{X,\beta\alpha,2}" description] &&\\
X^{\beta\alpha} \ar[rr, "\bi'_{X^\alpha,\beta}"'] && X^\alpha \ar[rrr, "\bi'_{X,\alpha}"'] &&& X.
\end{tikzcd}
\]
We have a sequence of natural maps or isomorphisms as follows:
\begin{multline*}
\bi_{X,\beta\alpha}^*(\bj_X)_*(\scF \otimes
f_\eta^*\scL_\bc)
\cong \bh_{X^\alpha,\beta}^* \bi_{X,\alpha,\beta}^* (\bj_{X,\beta\alpha,1})_*(\scF \otimes
f_\eta^*\scL_\bc) \\
\xrightarrow{\text{adjunction}} \bh_{X^\alpha,\beta}^* (\bj_{X^\alpha,\beta,1})_* (\bj_{X^\alpha,\beta,1})^* \bi_{X,\alpha,\beta}^* (\bj_{X,\beta\alpha,1})_*(\scF \otimes
f_\eta^*\scL_\bc) \\
\cong \bh_{X^\alpha,\beta}^* (\bj_{X^\alpha,\beta,1})_* \bh_{X,\alpha}^* (\bj_{X,\alpha,1})_*(\scF 
\otimes f_\eta^*\scL_\ba 
\otimes f_\eta^* \scL_{\bb}).
\end{multline*}
Now, recall the local system $\scL_{\bb,\alpha}$, and denote by $f_{\eta,\alpha} : X_{\eta,\alpha} \to \bA^P_{\eta,\alpha}$ the morphism induced by $f$. By adjunction and compatibility of pullback with tensor product, there exists a canonical morphism
\[
(\bj_{X,\alpha,1})_*(\scF \otimes
f_\eta^*\scL_\ba) \otimes
f_{\eta,\alpha}^* \scL_{\bb,\alpha} \to (\bj_{X,\alpha,1})_*(\scF \otimes
f_\eta^*\scL_\ba 
\otimes f_\eta^* \scL_{\bb}).
\]
This morphism becomes an isomorphism if $\scL_{\bb,\alpha}$ is replaced by the constant local system; since after pullback to $\overline{\F}$ the local system $\scL_{\bb,\alpha}$ is an extension of copies of this constant sheaf, we deduce that it is an isomorphism too. We deduce identifications
\begin{multline*}
\bh_{X^\alpha,\beta}^* (\bj_{X^\alpha,\beta,1})_* \bh_{X,\alpha}^* (\bj_{X,\alpha,1})_*(\scF \otimes f_\eta^*\scL_\ba \otimes f_\eta^* \scL_{\bb}) \\
\cong \bh_{X^\alpha,\beta}^* (\bj_{X^\alpha,\beta,1})_* \bh_{X,\alpha}^* \left( (\bj_{X,\alpha,1})_*(\scF \otimes f_\eta^*\scL_\ba) \otimes f_{\eta,\alpha}^* \scL_{\bb,\alpha} \right) \\
\cong \bh_{X^\alpha,\beta}^* (\bj_{X^\alpha,\beta,1})_* \left( \bigl( \bh_{X,\alpha}^* (\bj_{X,\alpha,1})_*(\scF \otimes f_\eta^*\scL_\ba) \bigr) \otimes \bigl( \bh_{X^\alpha,\beta}^* f_{\eta,\alpha}^* \scL_{\bb,\alpha} \bigr) \right) \\
\cong \bh_{X^\alpha,\beta}^* (\bj_{X^\alpha,\beta,1})_* \left( \bigl( \bh_{X,\alpha}^* (\bj_{X,\alpha,1})_*(\scF \otimes f_\eta^*\scL_\ba) \bigr) \otimes \bigl( (f^\alpha_{\eta})^* \bar\alpha_\eta^*\scL_{\bb,\alpha} \bigr) \right).
\end{multline*}
Using~\eqref{eqn:jordan-pullback} we deduce a canonical morphism
\begin{multline*}
\bh_{X^\alpha,\beta}^* (\bj_{X^\alpha,\beta,1})_* \bh_{X,\alpha}^* (\bj_{X,\alpha,1})_*(\scF \otimes f_\eta^*\scL_\ba \otimes f_\eta^* \scL_{\bb}) \\
\to \bh_{X^\alpha,\beta}^* (\bj_{X^\alpha,\beta,1})_* \left( \bigl( \bh_{X,\alpha}^* (\bj_{X,\alpha,1})_*(\scF \otimes f_\eta^*\scL_\ba) \bigr) \otimes (f^\alpha_{\eta})^* \scL_{\bb'} \right).
\end{multline*}
Using~\eqref{eqn:simplification-Upsilon}, Lemma~\ref{lem:degrees}, and the fact that tensoring with $(f^\alpha_{\eta})^* \scL_{\bb'}$ is exact for the perverse t-structure, applying perverse cohomology in degree $|R|-|P|$ to the composition of the maps above we deduce a canonical morphism
\begin{multline}
\label{eqn:morphism-composition}
\pH^{|R|-|P|}( \bi_{X,\beta\alpha}^*(\bj_X)_*(\scF \otimes f_\eta^*\scL_\bc)) \to \\
\pH^{|Q|-|R|} ( \bh_{X^\alpha,\beta}^* (\bj_{X^\alpha,\beta,1})_* ( \pH^{|P|-|Q|}( ( \bh_{X,\alpha}^* (\bj_{X,\alpha,1})_*(\scF \otimes f_\eta^*\scL_\ba) ) \otimes (f^\alpha_{\eta})^* \scL_{\bb'} ) )).
\end{multline}
When $\bc$ is large (among $\beta\alpha$-special maps) then $\ba$ is large (among $\alpha$-special maps) and $\bb'$ is large (among $\beta$-special maps) in view of~\eqref{eqn:b'-c}. Hence in this case~\eqref{eqn:morphism-composition} provides the morphism we were looking for.

\begin{rmk}
Suppose that $|P| = 2$, $|Q| = 1$, $R=\varnothing$, and $\alpha$ is nonconstant.  Then the morphism~\eqref{eqn:commnc-compose} is that appearing in~\cite[Lemma~9.4.3 or Lemma~9.4.11]{arbook}, depending on the size of $\alpha^{-1}(*)$.
\end{rmk}

\subsection{Compositions of higher nearby cycles: properties}

The following three statements give compatibility properties of the morphisms~\eqref{eqn:commnc-compose}. Each of them can be checked on definitions.

\begin{lem}\label{lem:commnc-assoc}
Suppose we have three pointed maps $\alpha: P_* \to Q_*$, $\beta: Q_* \to R_*$, and $\gamma: R_* \to S_*$.  If all the objects in the diagram below are defined, then the diagram commutes, where each arrow is an instance of~\eqref{eqn:commnc-compose}:
\[
\begin{tikzcd}[row sep=small]
\Upsilon^{\gamma\beta\alpha}_f(\scF) \ar[r] \ar[d] & \Upsilon^\gamma_{f^{\beta\alpha}}(\Upsilon^{\beta\alpha}_f(\scF)) \ar[d] \\
\Upsilon^{\gamma\beta}_{f^\alpha}(\Upsilon^\alpha_f(\scF)) \ar[r] & \Upsilon^\gamma_{f^{\beta\alpha}}(\Upsilon^\beta_{f^\alpha}(\Upsilon^\alpha_f(\scF))).
\end{tikzcd}
\]
\end{lem}

\begin{lem}
\label{lem:compat-compose-smooth}
Let $g: Y \to X$ be a smooth morphism of relative dimension $d$, let $\alpha: P_* \to Q_*$, $\beta: Q_* \to R_*$ be pointed maps, let $\scF \in \Perv(X_\eta,\bk)$. Assume that:
\begin{itemize}
\item
the $\alpha$-nearby cycles and the $\beta\alpha$-nearby cycles of $\scF$ are well defined;
\item
the $\beta$-nearby cycles of $\Upsilon_f^\alpha(\scF)$ are well defined.
\end{itemize}
Then the $\alpha$-nearby cycles and the $\beta\alpha$-nearby cycles of $g_\eta^* \scF[d]$, and the $\beta$-nearby cycles of $\Upsilon_{f g}^\alpha(g_\eta^* \scF[d])$, are all well defined, and the morphism
\[
\Upsilon_{f g}^{\beta\alpha}(g_\eta^* \scF[d]) \to 
\Upsilon^\beta_{(f g)^\alpha}(\Upsilon_{f g}^\alpha(g_\eta^* \scF[d]))
\]
given by~\eqref{eqn:commnc-compose}
is, taking into account the identifications of Lemma~\ref{lem:compat-smooth}, the image under $(g^\alpha_\eta)^*[d]$ of the corresponding morphism
$\Upsilon^{\beta \alpha}_f(\scF) \to \Upsilon^\beta_{f^\alpha}(\Upsilon_f^\alpha(\scF))$.\qed
\end{lem}

\begin{lem}
\label{lem:compat-compose-proper}
Let $g: Y \to X$ be a proper morphism, let $\alpha: P_* \to Q_*$, $\beta: Q_* \to R_*$ be pointed maps, and let $\scF \in \Perv(Y_\eta,\bk)$.  Assume that:
\begin{itemize}
\item
the $\alpha$-nearby cycles and the $\beta \alpha$-nearby cycles of $\scF$ are well defined;
\item
the $\beta$-nearby cycles of $\Upsilon_{fg}^\alpha(\scF)$ are well defined;
\item
the complexes
\[
(g_{\eta})_* \scF, \quad (g^\alpha_{\eta})_* \Upsilon^\alpha_{f g}(\scF), \quad (g^{\beta\alpha}_{\eta})_*\Upsilon^{\beta}_{(f g)^\alpha} (\Upsilon^{\alpha}_{f g}(\scF)) \quad \text{and} \quad (g^{\beta\alpha}_{\eta})_*\Upsilon^{\beta \alpha}_{f g}(\scF)
\]
are perverse.
\end{itemize}
Then the $\alpha$-nearby cycles and the $\beta\alpha$-nearby cycles of $(g_\eta)_* \scF$, and the $\beta$-nearby cycles of $\Upsilon_{f}^\alpha((g_\eta)_*\scF)$, are all well defined, and the morphism
\[
\Upsilon^{\beta \alpha}_{f}((g_\eta)_* \scF) \to \Upsilon^\beta_{f^\alpha}(\Upsilon_{f}^\alpha((g_\eta)_*\scF))
\]
given by~\eqref{eqn:commnc-compose}
is, taking into account the identifications of Lemma~\ref{lem:compat-proper}, the image under $(g^{\beta \alpha}_\eta)_*$ of the corresponding morphism
$\Upsilon^{\beta \alpha}_{f g}(\scF) \to \Upsilon^\beta_{(f g)^\alpha}(\Upsilon_{f g}^\alpha(\scF))$.\qed
\end{lem}

\subsection{Product-type situations}
\label{ss:product}

Let $P$ be a finite set, and suppose we have a collection of maps $(f_p: X_p \to \bA^1)_{p \in P}$.  For each $p$, let
\[
X_{p,\eta} = f_p^{-1}(\bA^1 \smallsetminus \{0\})
\qquad\text{and}\qquad
X_{p,0} = f_p^{-1}(\{0\}).
\]
Denote the inclusion maps by
$j_p: X_{p,\eta} \to X_p$ and
$i_p: X_{p,0} \to X_p$.
Set
\[
X = \prod_{p \in P} X_p \qquad\text{and}\qquad
f = \prod_{p \in P} f_p: X \to \bA^P.
\]
We obviously have $X_\eta = \prod_{p \in P} X_{p,\eta}$.  More generally, for any pointed map $\alpha: P_* \to Q_*$, we can describe $X^\alpha_\eta$ as follows:
\begin{equation}
\label{eqn:product-alpha}
X^\alpha_\eta \cong \prod_{q \in Q} X^\alpha_{q,\eta} \times \prod_{p \in \alpha^{-1}(*) \cap P} X_{p,0} \quad \text{where} \quad X^\alpha_{q,\eta} = \mathop{\prod\nolimits_{\bA^1}}\limits_{p \in \alpha^{-1}(q)} X_{p,\eta}.
\end{equation}
Here, the right-hand side is a fiber product over $\bA^1$.  If $\alpha^{-1}(q) = \varnothing$, the right-hand side should be understood to be $\bA^1$.

The following lemma is immediate from the definitions.

\begin{lem}\label{lem:product-calc1}
Let $(f_p: X_p \to \bA^1)_{p \in P}$ be as above. Suppose we have a collection of objects $\scF_p \in \Dbc(X_{p,\eta},\bk)$, and set
\[
\scF = \bigboxtimes_{p \in P} \scF_p \in \Dbc(X,\bk).
\]
Then the object $\bi_\alpha^*\bj_*\scF \in \Dbc(X^\alpha_\eta,\bk)$ is given by
\[
\bi_\alpha^*\bj_*\scF \cong
\bigboxtimes_{q \in Q} \left(\mathop{\bigboxtimes\nolimits_{\bA^1_\eta}}\limits_{p \in \alpha^{-1}(q)} \scF_p\right) \boxtimes
\bigboxtimes_{p \in \alpha^{-1}(*) \cap P} i_p^*j_{p*}\scF_p.
\]
\end{lem}

Here, the notation ``$\boxtimes_{\bA^1_\eta}$'' is a relative external tensor product: it is the pullback of the usual external tensor product $\bigboxtimes\limits_{p \in \alpha^{-1}(q)} \scF_p$ along the map
\[
\mathop{\prod\nolimits_{\bA^1}}\limits_{p \in \alpha^{-1}(q)} X_{p,\eta} \hookrightarrow
\prod_{p \in \alpha^{-1}(q)} X_{p,\eta}.
\]
(When $\alpha^{-1}(q) = \varnothing$, this is the map $\bA^1 \to \Spec(\F)$, and $\bigboxtimes\limits_{p \in \alpha^{-1}(q)} \scF_p$ should be understood to be the constant sheaf $\underline{\bk}$ on $\Spec(\F)$.)

\begin{lem}
\label{lem:product-upexist}
Let $(f_p: X_p \to \bA^1)_{p \in P}$ and $\alpha: P_* \to Q_*$ be as above.  Suppose we have a collection of perverse sheaves $\scF_p \in \Perv(X_{p,\eta},\bk)$ that satisfy the following condition: for each $q \in Q$, the object
\[
\left(\mathop{\bigboxtimes\nolimits_{\bA^1_\eta}}\limits_{p \in \alpha^{-1}(q)} \scF_p\right)[1-|\alpha^{-1}(q)|] \in \Dbc(X^\alpha_{q,\eta},\bk)
\]
is perverse.  Then the $\alpha$-nearby cycles of $\scF$ are well defined, and we have
\[
\Upsilon^\alpha_f(\scF) \cong 
\bigboxtimes_{q \in Q} \left(\mathop{\bigboxtimes\nolimits_{\bA^1_\eta}}\limits_{p \in \alpha^{-1}(q)} \scF_p\right)[1- |\alpha^{-1}(q)|] \boxtimes
\bigboxtimes_{p \in \alpha^{-1}(*) \cap P} \Psi^{\mathrm{un}}_{f_p}(\scF_p).
\]
\end{lem}
\begin{proof}
Consider two $\alpha$-special functions $\ba \le \bb$.  By Lemma~\ref{lem:product-calc1}, the map
\[
\bi_\alpha^*\bj_*(\scF \otimes f_\eta^*\scL_\ba) \to \bi_\alpha^*\bj_*(\scF \otimes f_\eta^*\scL_\bb)
\]
is the external tensor product of the following two kinds of maps:
\begin{align}
&\text{for $q \in Q$}: & &\mathop{\bigboxtimes\nolimits_{\bA^1_\eta}}\limits_{p \in \alpha^{-1}(q)} (\scF_p \otimes f_{p,\eta}^*\scL_{\ba(p)} \to \scF_p \otimes f_{p,\eta}^*\scL_{\bb(p)}); \label{eqn:pcalc1} \\
&\text{for $p \in \alpha^{-1}(*) \cap P$}: & & i_p^*j_{p*}(\scF_p \otimes f_{p,\eta}^*\scL_{\ba(p)} \to \scF_p \otimes f_{p,\eta}^*\scL_{\bb(p)}). \label{eqn:pcalc2}
\end{align}
In~\eqref{eqn:pcalc1}, because $\ba$ and $\bb$ are $\alpha$-special, we have $\ba(p)= \bb(p) = 1$ for each $p$ that appears.  That is,~\eqref{eqn:pcalc1} is just the identity map of the object $\mathop{\bigboxtimes\nolimits_{\bA^1_\eta}}\limits_{p \in \alpha^{-1}(q)} \scF_p$, which is a shifted perverse sheaf by assumption.

The perverse cohomology of~\eqref{eqn:pcalc2} is precisely Be{\u\i}linson's description of the unipotent nearby cycles of $\scF_p$.  More precisely, for $\ba$ and $\bb$ large enough, the $i$-th perverse cohomology of the map in~\eqref{eqn:pcalc2} is an isomorphism if $i=-1$, and is $0$ otherwise.

We conclude that, for $\ba$ and $\bb$ large enough, the perverse cohomology $\pH^i$ of the external tensor product of all the maps~\eqref{eqn:pcalc1} and~\eqref{eqn:pcalc2} is an isomorphism when
\[
i = \sum_{q \in Q} (1 - |\alpha^{-1}(q)|) + \sum_{p \in \alpha^{-1}(*) \cap P} (-1) 
= |Q| - |P|,
\]
and zero otherwise.
\end{proof}

\begin{lem}
\label{lem:composition-product}
Let $(f_p: X_p \to \bA^1)_{p \in P}$ be as above, and let $\alpha: P_* \to Q_*$ and $\beta: Q_* \to R_*$ be pointed maps.  Suppose we have a collection of perverse sheaves $\scF_p \in \Perv(X_{p,\eta},\bk)$ satisfying the following three conditions:
\begin{enumerate}
\item for each $q \in Q$, the following object is perverse:
\[
\left(\mathop{\bigboxtimes\nolimits_{\bA^1_\eta}}\limits_{p \in \alpha^{-1}(q)} \scF_p\right)[1-|\alpha^{-1}(q)|] \in \Dbc(X^\alpha_{q,\eta},\bk);
\]
\item for each $r \in R$, the following object is perverse:
\[
\left(\mathop{\bigboxtimes\nolimits_{\bA^1_\eta}}\limits_{p \in \alpha^{-1}(\beta^{-1}(r))} \scF_p\right)[1-|\alpha^{-1}(\beta^{-1}(r))|] \in \Dbc(X^{\beta \alpha}_{r,\eta},\bk);
\]
\item
for any $p \in P$ such that $\alpha(p) \neq *$ and $(\beta \alpha)(p)=*$ we have $\Psi^{\mathrm{un}}_{f_p}(\scF_p)=\Psi_{f_p}(\scF_p)$.
\end{enumerate}
Then the $\alpha$-nearby cycles and the $\beta\alpha$-nearby cycles of $\scF$ are well defined, as well as the $\beta$-nearby cycles of $\Upsilon^\alpha_f(\scF)$, and the  map
$\Upsilon^{\beta \alpha}_f(\scF) \to \Upsilon^\beta_{f^\alpha}(\Upsilon_f^\alpha(\scF))$
from~\eqref{eqn:commnc-compose} is an isomorphism.
\end{lem}

\begin{proof}
Our assumptions together with Lemma~\ref{lem:product-upexist} imply that the $\alpha$-nearby cycles and the $\beta\alpha$-nearby cycles of $\scF$ are well defined.
To study the $\beta$-nearby cycles of $\Upsilon^\alpha_f(\scF)$, let us introduce the notation $Z = \prod_{p \in \alpha^{-1}(*) \cap P} X_{p,0}$, so that $X^\alpha_\eta = \prod_{q \in Q} X^\alpha_{q,\eta} \times Z$. We also let
\[
\scG_q = 
\left(\mathop{\bigboxtimes\nolimits_{\bA^1_\eta}}\limits_{p \in \alpha^{-1}(q)} \scF_p\right)[1- |\alpha^{-1}(q)|]
\qquad\text{and}\qquad
\scG_Z = \bigboxtimes_{p \in \alpha^{-1}(*) \cap P} \Psi_{f_p}^{\mathrm{un}}(\scF_p),
\]
so that if we set $\scG = \Upsilon^\alpha_f(\scF)$, then by Lemma~\ref{lem:product-upexist} we have
\[
\scG \cong \left(\bigboxtimes_{q \in Q} \scG_q \right)\boxtimes \scG_Z.
\]

 The diagram
\[
X^\alpha_\eta \xrightarrow{\bj_{X^\alpha}} X^\alpha \xleftarrow{\bi_{X^\alpha,\beta}} X^{\beta\alpha}_\eta
\]
can be redrawn as
\[
\prod_{q \in Q} X^\alpha_{q,\eta} \times Z \xrightarrow{\bj_{X^\alpha}} \prod_{q \in Q} X^\alpha_q \times Z \xleftarrow{\bi_{X^\alpha,\beta}} \prod_{r \in R} \left( \mathop{\prod\nolimits_{\bA^1}}\limits_{r \in \beta^{-1}(q)} X^\alpha_{q,\eta}\right) \times \prod_{q \in \beta^{-1}(*) \cap Q} X^\alpha_{q,0} \times Z.
\]
This almost matches the general set-up at the beginning of this subsection, except for the extra factor of $Z$.  A minor variant of Lemma~\ref{lem:product-upexist} says that a sufficient condition for the $\beta$-nearby cycles of $\scG$ to be well defined is that for each $r \in R$ the object
\begin{equation}\label{eqn:beta-cond}
\left(\mathop{\bigboxtimes\nolimits_{\bA^1_\eta}}\limits_{q \in \beta^{-1}(r)} \scG_q\right)[1-|\beta^{-1}(r)|] 
\end{equation}
be perverse.  If this holds, then we have
\begin{equation}\label{eqn:beta-calc}
\Upsilon^\beta_{f^\alpha}(\scG) \cong 
\bigboxtimes_{r \in R} \left(\mathop{\bigboxtimes\nolimits_{\bA^1_\eta}}\limits_{q \in \beta^{-1}(r)} \scG_q\right)[1- |\beta^{-1}(r)|] \boxtimes
\bigboxtimes_{q \in \beta^{-1}(*) \cap Q} \Psi^{\mathrm{un}}_{f^\alpha_q}(\scG_q) \boxtimes \scG_Z.
\end{equation}

Using the definition of $\scG_q$, we rewrite the object in~\eqref{eqn:beta-cond} as
\begin{multline*}
\left(\mathop{\bigboxtimes\nolimits_{\bA^1_\eta}}\limits_{\substack{q \in \beta^{-1}(r)\\ p \in \alpha^{-1}(q)}} \scF_p\right)
\left[ 1 - |\beta^{-1}(r)| + \sum_{q \in \beta^{-1}(r)} (1 - |\alpha^{-1}(q)|) \right] \\
\cong
\left(\mathop{\bigboxtimes\nolimits_{\bA^1_\eta}}\limits_{p \in \alpha^{-1}(\beta^{-1}(r))} \scF_p\right)[1-|\alpha^{-1}(\beta^{-1}(r))|].
\end{multline*}
This object is perverse by the second assumption in the lemma.  We conclude that the $\beta$-nearby cycles of $\scG$ are well defined.  Moreover, using the definition of $\scG_q$, we can rewrite~\eqref{eqn:beta-calc} as
\[
\bigboxtimes_{r \in R}\left(\mathop{\bigboxtimes\nolimits_{\bA^1_\eta}}\limits_{\substack{q \in \beta^{-1}(r)\\ p \in \alpha^{-1}(q)}} \scF_p\right)[1-|\alpha^{-1}(\beta^{-1}(r))|] \boxtimes 
\bigboxtimes_{q \in \beta^{-1}(*) \cap Q} \Psi^{\mathrm{un}}_{f^\alpha_q}(\scG_q) \boxtimes \bigboxtimes_{p \in \alpha^{-1}(*) \cap P} \Psi^{\mathrm{un}}_{f_p}(\scF_p).
\]
To finish the comparison
with $\Upsilon^{\beta \alpha}_f(\scF)$, we must show that if $\beta(q) = *$, then
\[
\Psi^{\mathrm{un}}_{f_q^\alpha}(\scG_q) \cong \bigboxtimes_{p \in \alpha^{-1}(q)} \Psi^{\mathrm{un}}_{f_p}(\scF_p).
\]
This claim follows from the definition of $\scG_q$, our third assumption, and the compatibility of nearby cycles with external tensor products, see~\cite[Th\'eor\`eme~4.7]{illusie-autour}.
\end{proof}

\section{Application to central sheaves}
\label{sec:app}

In this section we assume (for simplicity) that $\F$ is algebraically closed, and denote by $\Alg_\F$ the category of $\F$-algebras.

\subsection{Graphs of points}

We consider the curve $C=\mathbb{A}^1_\F$ and the closed point $0 \in C(\F)$.
Given $R \in \Alg_\F$ and $y \in C(R)$, we denote by $\Gamma_y \subset C_R := C \times_{\Spec(\F)} \Spec(R)$ the graph of $y$. The constant $R$-point defined by $0 \in C(\F)$ will also be denoted $0$. If $P$ is a finite set,
for $R \in \Alg_\F$ and $(y_p)_{p \in P} \in C^P(R)$ we set
\[
\Gamma_{\{y_p : p \in P\}} = \bigcup_{p \in P} \Gamma_{y_p} \quad \subset C_R.
\]
Of course, for any subset $Q \subset P$ we have a closed immersion
$\Gamma_{\{y_p : p \in Q\}} \to \Gamma_{\{y_p : p \in P\}}$.
We will also denote by $\hGamma_{\{y_p : p \in P\}}$ the completion of $C_R$ along $\Gamma_{\{y_p : p \in P\}}$ (i.e.~the spectrum of the completion of $\scO(C_R)$ with respect to the ideal of definition of the closed subscheme $\Gamma_{\{y_p : p \in P\}}$). We have a natural morphism
$\hGamma_{\{y_p : p \in P\}} \to C_R$,
and the closed immersion $\Gamma_{\{y_p : p \in P\}} \to C_R$ factors through a closed immersion
$\Gamma_{\{y_p : p \in P\}} \to \hGamma_{\{y_p : p \in P\}}$.
We set
\[
\hGamma_{\{y_p : p \in P\}}^\circ := \hGamma_{\{y_p : p \in P\}} \smallsetminus \Gamma_{\{y_p : p \in P\}}.
\]

\subsection{Satake category and central sheaves}
\label{ss:Satake-category}

Let $G$ be a connected reductive algebraic group over $\F$.
To $G$ and a choice of Borel subgroup $B \subset G$ we can associate in the usual way the loop group $LG$, the positive loop group $L^+G$, the Iwahori subgroup $I \subset L^+G$, the affine Grassmannian $\Gr_G=LG/L^+G$ and the affine flag variety $\Fl_G=LG/I$. Here the quotients are the fppf quotients, and they are represented by ind-projective ind-schemes over $\F$; for all of this, see~\cite{arbook} for details. Recall that the $L^+G$-equivariant derived category $\Db_{L^+G}(\Gr_G,\bk)$ of $\bk$-sheaves on $\Gr_G$, resp.~the $I$-equivariant derived category $\Db_{I}(\Fl_G,\bk)$ of $\bk$-sheaves on $\Fl_G$, is endowed with a canonical unital and associative convolution product $\star^{L^+G}$, resp.~$\star^I$.

The Satake category is the category $\Perv_{L^+G}(\Gr_G,\bk)$ of $L^+G$-equivariant $\bk$-perverse sheaves on $\Gr_G$. 
It is a fundamental standard fact (see~\cite{mv,br}) that the product $\star^{L^+G}$ is t-exact on both sides, hence restricts to a bifunctor on the Satake category, and moreover that this restriction admits a canonical commutativity constraint. For a finite collection $(\scA_p)_{p \in P}$ of objects in $\Perv_{L^+G}(\Gr_G,\bk)$, it therefore makes sense to consider the convolution product $\bigstar_{p \in P}^{L^+G} \scA_p$.

We will denote by $\cG$ the smooth affine group scheme over $C$ constructed (following X.~Zhu) in~\cite[\S 2.2.3.1]{arbook}: its restriction to $C \smallsetminus \{0\}$, resp.~to the formal neighborhood of $0$, identifies with $G \times (C \smallsetminus \{0\})$, resp.~with the Iwahori group scheme of $LG$ attached to $B$. For any scheme $X$ over $C$, we will denote by $\cE^0_X = X \times_C \cG$ the trivial principal $\cG$-bundle over $X$. 
 Recall the ind-scheme $\bGr^{\mathrm{Cen}}_\cG$ over $C$ defined in~\cite[\S 2.2.3.2]{arbook}; it represents the functor sending $R \in \Alg_\F$ to the set of equivalence classes of triples $(y,\cE,\beta)$ where:
\begin{itemize}
\item
$y \in C(R)$;
\item
$\cE$ is a principal $\cG$-bundle over $\hGamma_y$;
\item
$\beta : \cE_{|\hGamma_y^\circ} \simto \cE^0_{\hGamma_y^\circ}$ is an isomorphism.
\end{itemize}
We have canonical identifications
\begin{equation}
\label{eqn:bGr-Cen-fibers}
\bGr^{\mathrm{Cen}}_\cG{}_{|\{0\}} \cong \Fl_G, \quad
\bGr^{\mathrm{Cen}}_\cG{}_{|C \smallsetminus \{0\}} \cong \Gr_G \times (C \smallsetminus \{0\}).
\end{equation}
Following Gaitsgory~\cite{gaitsgory}, we consider the functor
\[
 \sfZ : \Perv_{L^+G}(\Gr_G,\bk) \to \Perv_I(\Fl_G,\bk)
\]
defined by
$\sfZ(\scA)= \Upsilon_{\bGr^{\mathrm{Cen}}_\cG}(\scA \boxtimes \underline{\bk}_{C \smallsetminus \{0\}} [1])$.
In fact, in this setting it is known that the nearby cycles of $\scA \boxtimes \underline{\bk}_{C \smallsetminus \{0\}} [1]$ are unipotent (see~\cite[\S 2.4.5]{arbook}), so that $\sfZ(\scA)$ coincides with the full nearby cycles, see Example~\ref{ex:Q-empty}. It is known that this functor is monoidal when seen as a functor with values in $\Db_I(\Fl_G,\bk)$ (see~\cite[Theorem~3.4.1]{arbook}), that for any $\scF$ in $\Perv_I(\Fl_G,\bk)$ and $\scA$ in $\Perv_{L^+G}(\Gr_G,\bk)$ the convolution $\scF \star^I \sfZ(\scA)$ is perverse (see~\cite[Corollary~3.2.5]{arbook}), and that $\sfZ$ is a central functor; in particular, for $\scF$, $\scA$ as above there exists a canonical isomorphism $\scF \star^I \sfZ(\scA) \cong \sfZ(\scA) \star^I \scF$, see~\cite[Theorem~3.2.3 and~\S 3.5.1]{arbook}. In particular, for a finite collection $(\scA_p)_{p \in P}$ of objects in $\Perv_{L^+G}(\Gr_G,\bk)$, it makes sense to consider the convolution product $\bigstar_{p \in P}^{I} \sfZ(\scA_p)$.

\subsection{Iterated affine Grassmannians}

Let $P$ be a finite set.  Define a functor $\bGr_P$ on $\Alg_\F$ as follows: for $R \in \Alg_\F$, $\bGr_P(R)$ is the set of equivalence classes of the following data:
\begin{itemize}
\item a point $(y_p)_{p \in P}$ in $C^P(R)$;
\item
a principal $\cG$-bundle $\cE$ over $\hGamma_{\{0\} \cup \{y_p: p \in P\}}$;
\item
an isomorphism $\beta : \cE_{| \hGamma^\circ_{\{0\} \cup \{y_p: p \in P\}}} \simto \cE^0_{\hGamma_{\{0\} \cup \{y_p: p \in P\}}^\circ}$.
\end{itemize}
This functor is represented by an ind-proper ind-scheme over $C^P$. It is also easily seen that if $Q$ is another finite set and $\alpha : P_* \to Q_*$ is a surjective pointed map, there is a canonical identification $\bA^Q \times_{\bA^P} \bGr_P = \bGr_Q$. 

\begin{ex}
For $n \in \Z_{\geq 1}$ and $P=\{1, \ldots, n\}$, the ind-scheme $\bGr_{\{1, \ldots, n\}}$ coincides with the ind-scheme $\bGr_n$ of~\cite[\S 5.1]{bbfk}. If $P=\varnothing$ we have $\bGr_\varnothing=\Fl_G$.
\end{ex}

Denote by
\[
C^{P,\dag} \subset C^P
\]
the open subscheme consisting of the points $(y_p)_{p \in P}$ such that $y_p \neq 0$ for any $p$ and $y_p \neq y_{p'}$ for any $p \ne p'$.  By standard arguments we have a canonical identification
\begin{equation}
\label{eqn:identification-GrP-generic}
(\bGr_P)_{|C^{P,\dag}} \cong \Fl_G \times (\Gr_G)^P \times C^{P,\dag}.
\end{equation}
Denote by $\jmath_P$ the open embedding
\[
(\bGr_P)_{|C^{P,\dag}} \to (\bGr_P)_{|(C \smallsetminus \{0\})^P} = (\bGr_P)_\eta.
\]

Below we will consider collections of perverse sheaves $\scA_* \in \Perv_I(\Fl_G,\bk)$ and $\scA_p \in \Perv_{L^+G}(\Gr_G,\bk)$ for each $p \in P$. For brevity, we denote this collection by
$(\scA_i)_{i \in P_*}$.
We consider the functor
\[
\sfC_P : \Perv_I(\Fl_G,\bk) \times \prod_{p \in P} \Perv_{L^+G}(\Gr_G,\bk) \to \Perv((\bGr_P)_\eta)
\]
defined by
\[
\sfC_P((\scA_i)_{i \in P_*}) = (\jmath_P)_{!*}\left(\scA_* \boxtimes \left( \bigboxtimes_{p \in P} \scA_p \right) \boxtimes \underline{\bk}_{C^{P,\dag}}[|P|]\right),
\]
where we use the identification~\eqref{eqn:identification-GrP-generic}.

The main result of this section is the following statement, proved in~\S\ref{ss:proof-new}.  The statement involves 
the extension of the constructions of Section~\ref{sec:higher-nc} to ind-schemes of ind-finite type, see Remark~\ref{rmk:def-Upsilon}\eqref{it:ind-schemes}.

\begin{thm}
\label{thm:nc-main-new}
Let $\alpha: P_* \to Q_*$ be a surjective pointed map.  For any $\scA_*$ in $\Perv_I(\Fl_G,\bk)$ and $(\scA_p)_{p \in P}$ in $\Perv_{L^+G}(\Gr_G,\bk)$, the $\alpha$-nearby cycles of $\sfC_P((\scA_i)_{i \in P_*})$ are well defined, and moreover we have a canonical identification
\[
\Upsilon^\alpha_{\bGr_P} \bigl( \sfC_P((\scA_i)_{i \in P_*}) \bigr)
\cong \sfC_Q \bigl( (\scB_j)_{j \in Q_*} \bigr)
\]
where
\[
\scB_* = \scA_* \star^I \left(\bigstar_{p \in \alpha^{-1}(*) \cap P}^I \sfZ(\scA_p)\right)
\qquad\text{and}\qquad
\scB_q = \bigstar_{p \in \alpha^{-1}(q)}^{L^+G} \scA_p
\quad\text{for $q \in Q$.}
\]
If $\beta: Q_* \to R_*$ is another surjective pointed map, then the natural map
\[
\Upsilon^{\beta\alpha}_{\bGr_P} \bigl( \sfC_P((\scA_i)_{i \in P_*}) \bigr) \to \Upsilon^\beta_{\bGr_Q}\Upsilon^\alpha_{\bGr_P} \bigl( \sfC_P((\scA_i)_{i \in P_*}) \bigr)
\]
given by~\eqref{eqn:commnc-compose}
is an isomorphism.
\end{thm}

\begin{rmk}
 As was explained to us by A. Salmon, Theorem~\ref{thm:nc-main-new} can be restated as the construction of a category cofibered over the category of finite pointed sets (and pointed maps).
\end{rmk}

\begin{rmk}
\label{rmk:thm-special}
In this remark we assume that $P=\{1, \ldots, n\}$ for some $n \in \Z_{\geq 1}$ and $Q=\varnothing$. In this case there is a unique choice for $\alpha$, we have $\bGr_Q=\Fl_G$, and Theorem~\ref{thm:nc-main-new} says that
\begin{equation}
\label{eqn:nearby-cycles-Gr}
\Upsilon^\alpha_{\bGr_P}(\scA_*, \scA_1, \ldots, \scA_n) \cong \scA_* \star^I \sfZ(\scA_1) \star^I \cdots \star^I \sfZ(\scA_n).
\end{equation}
\begin{enumerate}
\item
If $n=1$, by Example~\ref{ex:Q-empty} the fact that the nearby cycles are well defined is automatic; the isomorphism~\eqref{eqn:nearby-cycles-Gr} is the content of~\cite[Proposition~3.2.1]{arbook}. In case $n=2$, this statement is closely related to the results of~\cite[\S 3.5]{arbook}.\label{it:thm-special-1}
\item
There exists a natural action of the symmetric group $\mathfrak{S}_n$ on $\bGr_P$ by permutation of the points $y_i$. This action preserves the preimage of $C^{P,\dag}$ and, under the identification~\eqref{eqn:identification-GrP-generic}, its restriction to this open subset identifies with the diagonal action by permutation of the factors in $(\Gr_G)^n$ and $C^{P,\dag}$. It also preserves the preimage of $(0, \ldots,0)$ and restricts to the identity on this preimage. For any $\sigma \in \mathfrak{S}_n$ we deduce a canonical isomorphism between $\Upsilon_{\bGr_P}(\sfC_P(\scA_*, \scA_1, \ldots, \scA_n))$ and the similar object obtained by permutation of the $\scA_i$'s. Using the same techniques as in~\cite[\S 3.5.8]{arbook} one can check that, under~\eqref{eqn:nearby-cycles-Gr}, this isomorphism is induced by the ``centrality'' isomorphism for the functor $\sfZ$ (see~\cite[Theorem~3.2.3]{arbook}) or, equivalently (by~\cite{gaitsgory-app}, see~\cite[Theorem~3.5.1]{arbook}), by the commutativity constraint on the Satake category.
\end{enumerate}
\end{rmk}

\subsection{Convolution--torsor affine Grassmannians}
\label{ss:conv-torsor}

We now introduce some auxiliary ind-schemes needed for the proof of Theorem~\ref{thm:nc-main-new}.  In this section, we 
assume that $P_*$ and $Q_*$ are equipped with total orders such that $*$ is the smallest element, and 
that
$\alpha: P_* \to Q_*$
is a surjective, order-preserving pointed map.  
We set
\[
\min(\alpha^{-1}) := \{ i \in P \mid i = \min(\alpha^{-1}(\alpha(i))) \}, \quad \nmin(\alpha^{-1}) := P_* \smallsetminus \min(\alpha^{-1}).
\]
For $i$ in $P$ or $Q$, we will denote by $i-1$ the predecessor of $i$.

Let $\bc$ and $\bt$ be two subsets of $P$ such that
$\bc \cap \bt = \varnothing$.
We will call $\bc$ the ``convolution locus,'' and $\bt$ the ``torsor locus.''  (These terms will be justified below.). Define a functor $\ctGr_\alpha^{\bc,\bt}$ as follows: for $R \in \Alg_\F$, $\ctGr_\alpha^{\bc,\bt}(R)$ is the set of equivalence classes of the following data:
\begin{itemize}
\item a point $(y_q)_{q \in Q}$ in $C^Q(R)$;
\item
for $i \in P_*$, a
principal $\cG$-bundle $\cE^i$ over $\hGamma_{\{0\} \cup \{y_q: q \in Q\}}$;
\item 
for $i \in P_* \smallsetminus \bc$, 
an isomorphism 
\[
\beta^i : \cE^i_{| \hGamma_{\{0\} \cup \{y_q: q \in Q\}} \smallsetminus \Gamma_{y_{\alpha(i)}}} \simto \cE^0_{\hGamma_{\{0\} \cup \{y_q: q \in Q\}} \smallsetminus \Gamma_{y_{\alpha(i)}}}.
\]
\item 
for $i \in \bc$, an isomorphism $\beta^i : \cE^i_{| \hGamma_{\{0\} \cup \{y_q: q \in Q\}} \smallsetminus \Gamma_{y_{\alpha(i)}}} \simto \cE^{i-1}_{|\hGamma_{\{0\} \cup \{y_q: q \in Q\}} \smallsetminus \Gamma_{y_{\alpha(i)}}}$;
\item 
for $i \in \bt$, an isomorphism $\gamma^i : \cE^{i-1} \simto \cE^0_{\hGamma_{\{0\} \cup \{y_q: q \in Q\}}}$.
\end{itemize}
In this definition,
if $\alpha(i) = *$, then ``$y_{\alpha(i)}$'' should be taken to mean the point $0 \in C(R)$.
In the special case where $\alpha$ is the identity map, we may write $\ctGr_P^{\bc,\bt}$ instead of $\ctGr_\alpha^{\bc,\bt}$. Using standard arguments (see e.g.~\cite[Proposition~2.3.11]{arbook}) one can show that $\ctGr_\alpha^{\bc,\bt}$ is represented by an ind-scheme over $C^Q$, which is moreover ind-proper if $\bt=\varnothing$.

\begin{ex}
For $n \in \Z_{\geq 1}$ and $P=\{1, \ldots, n\}$, the ind-scheme $\ctGr_{\{1, \ldots, n\}}^{\{1, \ldots, n\},\varnothing}$ coincides with the ind-scheme $\tbGr_n$ of~\cite[\S 5.1]{bbfk}.
\end{ex}

If $\bt' \subset \bt$, there is an obvious map
\begin{equation}
\label{eqn:def-q}
q: \ctGr_\alpha^{\bc,\bt} \to \ctGr_\alpha^{\bc,\bt \smallsetminus \bt'}
\end{equation}
given by forgetting the $\gamma^i$'s with $i \in \bt'$.  There is also a ``twisting map''
\begin{equation}
\label{eqn:def-p}
p: \ctGr_\alpha^{\bc,\bt} \to \ctGr_\alpha^{\bc \cup \bt', \bt \smallsetminus \bt'}
\end{equation}
that is defined on $R$-points as follows: for each $j \in \bt'$, replace $\beta^j$ by the composition
\[
\cE^j_{| \hGamma_{\{0\} \cup \{y_q: q \in Q\}} \smallsetminus \Gamma_{y_{\alpha(j)}}} \xrightarrow{\beta^j} \cE^0_{\hGamma_{\{0\} \cup \{y_q: q \in Q\}} \smallsetminus \Gamma_{y_{\alpha(j)}}} \xrightarrow{(\gamma^j)^{-1}} \cE^{j-1}_{| \hGamma_{\{0\} \cup \{y_q: q \in Q\}} \smallsetminus \Gamma_{y_{\alpha(j)}}} ,
\]
and then forget $\gamma^j$.  

Let us describe this ind-scheme (or its generic part) in some special cases.  First, when $\bc = \bt = \varnothing$, we have
\begin{equation}\label{eqn:ctGr-product}
\ctGr_\alpha^{\varnothing,\varnothing} \cong
\underbrace{\Fl_G \times \cdots \times \Fl_G}_{\text{$|\alpha^{-1}(*)|$ copies}} \times
\prod_{j \in Q} (\underbrace{\bGr^{\mathrm{Cen}}_\cG \times_C \cdots \times_C \bGr^{\mathrm{Cen}}_\cG)}_{\text{$|\alpha^{-1}(j)|$ copies}}.
\end{equation}
In particular, its generic part is
\begin{equation}\label{eqn:ctGr-gen}
(\ctGr_\alpha^{\varnothing,\varnothing})_\eta
\cong
\underbrace{\Fl_G \times \cdots \times \Fl_G}_{\text{$|\alpha^{-1}(*)|$ copies}}
\times
\underbrace{\Gr_G \times \cdots \times \Gr_G}_{\text{$|\alpha^{-1}(Q)|$ copies}}
\times (C \smallsetminus \{0\})^Q.
\end{equation}
Next, suppose $\bc=\nmin(\alpha^{-1})$.
We have
\begin{multline}
\label{eqn:ctGr-conv}
(\ctGr_\alpha^{\nmin(\alpha^{-1}),\varnothing})_\eta
\cong
\underbrace{LG \times^I LG \times^I \cdots \times^I \Fl_G}_{\text{$|\alpha^{-1}(*)|$ factors}}
\times \\
\prod_{j \in Q}
\underbrace{LG \times^{L^+G} LG \times^{L^+G} \cdots \times^{L^+G} \Gr_G}_{\text{$|\alpha^{-1}(j)|$ factors}} \times (C \smallsetminus \{0\})^Q.
\end{multline}
More generally, the previous description remains valid over $C^{Q,\dag}$ for any $\bc$ containing $\nmin(\alpha^{-1})$:
\begin{equation}
\label{eqn:ctGr-conv-dag}
(\ctGr_\alpha^{\bc,\varnothing})_{|C^{Q,\dag}}
\cong
(\ctGr_\alpha^{\nmin(\alpha^{-1}),\varnothing})_{|C^{Q,\dag}}
\qquad
\text{if $\bc \supset \nmin(\alpha^{-1})$.}
\end{equation}
However, over a point $(y_q)_{q \in Q} \notin C^{Q,\dag}$, the fiber of $\ctGr_\alpha^{\bc,\varnothing}$ may differ from~\eqref{eqn:ctGr-conv} in the following way: some instances of ``$\Gr_G \times ({-})$'' are replaced by ``$LG \times^{L^+G} ({-})$,'' depending on $\bc$ and on the coincidences among the $y_j$'s. 

We now explain why $\bt$ is called the ``torsor locus.''  Define the pro-smooth group scheme $\cL_Q^+ \cG$ over $C^Q$ which represents the functor on $\Alg_\F$ such that $(\cL_Q^+ \cG)(R)$ consists of the tuples $((y_q)_{q \in Q}, g)$ with $(y_q)_{q \in Q} \in C^Q(R)$ and $g \in \cG(\hGamma_{\{0\} \cup \{y_q : q \in Q\}})$. (The representability of this group scheme can be proved as in~\cite[\S 3.5.2]{arbook}.) The following lemma follows from standard arguments (see e.g.~\cite[Lemma~2.3.9]{arbook}).

\begin{lem}\label{lem:ctGr-torsor}
The maps~\eqref{eqn:def-q} and~\eqref{eqn:def-p}
are both principal bundles (with respect to different actions) for the group scheme
\[
\mathop{\prod\nolimits_{C^Q}}\limits_{i \in \bt'} \cL^+_Q\cG.
\]
\end{lem}

Suppose we have a collection of perverse sheaves $(\scA_i)_{i \in P_*}$, where
\begin{equation}\label{eqn:perv-data}
\scA_i \in \Perv_I(\Fl_G,\bk)\ \text{if $\alpha(i) = *$,} 
\qquad
\scA_i \in \Perv_{L^+G}(\Gr_G,\bk)\ \text{if $\alpha(i) \in Q$.}
\end{equation}
Via~\eqref{eqn:ctGr-gen}, regard $(\bigboxtimes_i \scA_i) \boxtimes \underline{\bk}_{(C \times \{0\})^Q}[|Q|]$ as a perverse sheaf on $(\ctGr_\alpha^{\varnothing,\varnothing})_\eta$.  Next, fix some subset $\bc \subset P$, and consider the maps
\begin{equation}\label{eqn:convdiag-pre}
\ctGr_\alpha^{\varnothing,\varnothing} \xleftarrow{q}
\ctGr_\alpha^{\varnothing,\bc} \xrightarrow{p} \ctGr_\alpha^{\bc,\varnothing}.
\end{equation}
By equivariant descent, there is a unique object
\[
\tsfC_\alpha^\bc((\scA_i)_{i \in P_*}) \in \Perv((\ctGr_\alpha^{\bc,\varnothing})_\eta,\bk)
\]
such that we have an isomorphism
\[
p_\eta^*\tsfC_\alpha^\bc((\scA_i)_{i \in P_*}) \cong q_\eta^*\left(\left(\bigboxtimes_{i \in P_*} \scA_i\right) \boxtimes \underline{\bk}_{(C \times \{0\})^Q}[|Q|]\right).
\]
As an example, in the special case where $\bc = \nmin(\alpha^{-1})$, using the identification from~\eqref{eqn:ctGr-conv}, we have
\begin{multline}\label{eqn:tsfC-alpha}
\tsfC_\alpha^{\nmin(\alpha^{-1})}((\scA_i)_{i \in P_*}) \cong 
(\scA_* \tboxtimes \cdots \tboxtimes \scA_{\max(\alpha^{-1}(*))})
\boxtimes \\
\left( \bigboxtimes_{j \in Q} (\scA_{\min(\alpha^{-1}(j))} \tboxtimes \cdots \tboxtimes \scA_{\max(\alpha^{-1}(j))}) \right)
\boxtimes \underline{\bk}_{(C \smallsetminus \{0\})^Q}[|Q|].
\end{multline}
(Here, $\tboxtimes$ denotes the usual twisted external product.)
More generally, thanks to~\eqref{eqn:ctGr-conv-dag}, the previous description remains valid over $C^{Q,\dag}$ for any $\bc$ containing $\nmin(\alpha^{-1})$:
\begin{equation}\label{eqn:tsfC-alpha-dag}
\tsfC_\alpha^{\bc}((\scA_i)_{i \in P_*})_{|C^{Q,\dag}} \cong \tsfC_\alpha^{\nmin(\alpha^{-1})}((\scA_i)_{i \in P_*})_{|C^{Q,\dag}}.
\end{equation}

\subsection{Two kinds of convolution}

We continue with the setting of~\S\ref{ss:conv-torsor}, and
assume that $\bc \supset \nmin(\alpha^{-1})$.  This implies that $\bt \subset \min(\alpha^{-1})$.  We define a map
\[
m = m_\alpha^{\bc,\bt}: \ctGr_\alpha^{\bc ,\bt} \to \ctGr_Q^{\alpha(\bc \cap \min(\alpha^{-1})), \alpha(\bt)}
\]
by sending an $R$-point $((y_j), (\cE^i), (\beta^i), (\gamma^i))$ 
to $((y_j), (\cF^j), (\tilde\beta^j), (\tilde\gamma^j))$ where
\[
\cF^j := \cE^{\max(\alpha^{-1}(j))}, \quad
\tilde\beta^j := \beta^{\min(\alpha^{-1}(j))} \circ \cdots \circ 
\beta^{\max(\alpha^{-1}(j))}, \quad
\tilde\gamma^j = \gamma^{\min(\alpha^{-1}(j))}.
\]
To check that these definitions make sense, let us record the domains and codomains of the various maps above.  Because $\bc \supset \nmin(\alpha^{-1})$, we have
\begin{gather*}
\beta^i : \cE^i_{| \hGamma_{\{0\} \cup \{y_q: q \in Q\}} \smallsetminus \Gamma_{y_j}} \simto \cE^{i-1}_{|\hGamma_{\{0\} \cup \{y_q: q \in Q\}} \smallsetminus \Gamma_{y_j}} \text{ if $\min(\alpha^{-1}(j)) < i \le \max(\alpha^{-1}(j))$,} \\
\beta^{\min(\alpha^{-1}(j))} :  \cE^i_{| \hGamma_{\{0\} \cup \{y_q: q \in Q\}} \smallsetminus \Gamma_{y_j}} \simto \cF^{j-1}_{|\hGamma_{\{0\} \cup \{y_q: q \in Q\}} \smallsetminus \Gamma_{y_j}} \text{ if $\min(\alpha^{-1}(j)) \in \bc$,} \\
\beta^{\min(\alpha^{-1}(j))} :  \cE^i_{| \hGamma_{\{0\} \cup \{y_q: q \in Q\}} \smallsetminus \Gamma_{y_j}} \simto \cE^0_{\hGamma_{\{0\} \cup \{y_q: q \in Q\}} \smallsetminus \Gamma_{y_j}} \text{ if $\min(\alpha^{-1}(j)) \notin \bc$,} \\
\gamma^{\min(\alpha^{-1}(j))} : \cF^{j-1} \simto \cE^0_{\hGamma_{\{0\} \cup \{y_q: q \in Q\}}} \text{ if $\min(\alpha^{-1}(j)) \in \bt$.} 
\end{gather*}

In the special case where $\bc = \nmin(\alpha^{-1})$ and $\bt = \varnothing$, this map can be combined with~\eqref{eqn:convdiag-pre} to obtain the following ``convolution diagram'': 
\begin{equation}\label{eqn:convdiag-alpha}
\ctGr_\alpha^{\varnothing,\varnothing} \xleftarrow{q}
\ctGr_\alpha^{\varnothing,\nmin(\alpha^{-1})} \xrightarrow{p} \ctGr_\alpha^{\nmin(\alpha^{-1}),\varnothing}
\xrightarrow{m} \ctGr_Q^{\varnothing,\varnothing}.
\end{equation}
The following lemma is immediate from~\eqref{eqn:tsfC-alpha}.

\begin{lem}\label{lem:conv-calc1}
Let $(\scA_i)_{i \in P_*}$ be as in~\eqref{eqn:perv-data}.  There is a canonical isomorphism
\begin{multline*}
(m_\eta)_*\tsfC_\alpha^{\nmin(\alpha^{-1})}((\scA_i)_{i \in P_*}) \cong 
(\scA_* \star^I \cdots \star^I \scA_{\max(\alpha^{-1}(*))})
\boxtimes \\
\left( \bigboxtimes_{j \in Q} (\scA_{\min(\alpha^{-1}(j))} \star^{L^+G} \cdots \star^{L^+G} \scA_{\max(\alpha^{-1}(j))}) \right) \boxtimes \underline{\bk}_{(C \smallsetminus \{0\})^Q} [|Q|].
\end{multline*}
\end{lem}

Next, we define a map
\[
\mu = \mu_\alpha: \ctGr_\alpha^{P, \varnothing} \to \bGr_Q
\]
that sends an $R$-point $((y_j), (\cE^i), (\beta^i))$ to $((y_j), \cE^{\max(P)}, \hat\beta)$ where 
\[
\hat\beta = \beta^*_{| \hGamma^\circ_{\{0\} \cup \{y_q: q \in Q\}}} 
\circ \cdots \circ \beta^{\max(P)-1}_{| \hGamma^\circ_{\{0\} \cup \{y_q: q \in Q\}}} \circ \beta^{\max(P)}_{| \hGamma^\circ_{\{0\} \cup \{y_q: q \in Q\}}}.
\]
We combine this with~\eqref{eqn:convdiag-pre} to obtain a second convolution diagram
\begin{equation}\label{eqn:convdiag-mu}
\ctGr_\alpha^{\varnothing,\varnothing} \xleftarrow{q}
\ctGr_\alpha^{\varnothing,P} \xrightarrow{p} \ctGr_\alpha^{P,\varnothing}
\xrightarrow{\mu_\alpha} \bGr_Q.
\end{equation}

\begin{lem}
\label{lem:conv-calc2}
Let $(\scA_i)_{i \in P_*}$ be as in~\eqref{eqn:perv-data}.  There is a canonical isomorphism
\[
(\mu_{\alpha,\eta})_*\tsfC_\alpha^{P} \bigl( (\scA_i)_{i \in P_*} \bigr)
\cong \sfC_Q \bigl( (\scB_j)_{j \in Q_*} \bigr)
\]
where the $\scB_j$'s are as in Theorem~\ref{thm:nc-main-new}.
\end{lem}

\begin{proof}
Let us first treat the special case where $P = Q$ and $\alpha$ is the identity map.  In this case, the statement of the lemma simplifies to
\begin{equation}
\label{eqn:conv-calc-special}
(\mu_{\mathrm{id}_P, \eta})_*\tsfC_{\mathrm{id}_P}^{P}((\scA_i)_{i \in P_*}) \cong \sfC_P((\scA_i)_{i \in P_*}).
\end{equation}
The proof in this case is similar to that of~\cite[Lemma~1.7.10]{br} (the crucial step in the comparison of fusion and convolution in the Satake category). As a first step, we deduce from~\eqref{eqn:tsfC-alpha-dag} that
\[
\left( (\mu_{\mathrm{id}_P, \eta})_*\tsfC_{\mathrm{id}_P}^{P}((\scA_i)_{i \in P_*}) \right)_{| C^{P,\dag}} \cong \scA_* \boxtimes \left( \bigboxtimes_{p \in P} \scA_p \right) \boxtimes \underline{\bk}_{C^{P,\dag}}[|P|].
\]
We wish to prove that $(\mu_{\mathrm{id}_P, \eta})_*\tsfC_{\mathrm{id}_P}^{P}((\scA_i)_{i \in P_*})$ is the intermediate extension of the object above.  To do this, we use the standard characterization of the intermediate extension from~\cite[Proposition~2.1.9]{bbd} or~\cite[Lemma~3.3.4]{achar-book}: namely, it suffices to prove that the restriction, resp.~corestriction, of $(\mu_{\mathrm{id}_P, \eta})_*\tsfC_{\mathrm{id}_P}^{P}((\scA_i)_{i \in P_*})$ to the complement of $C^{P,\dag}$ in $(C \smallsetminus \{0\})^P$ lies in perverse degrees $\leq -1$, resp.~$\geq 1$. One can stratify $(C \smallsetminus \{0\})^P$ in terms of coincidences between points, with strata indexed by partitions of $P$. Given a partition $\tau$ into $m$ subsets, the preimage of the stratum $X_\tau$ attached to $\tau$ (of dimension $m$) in $\bGr_P$ identifies with $\Fl_G \times (\Gr_G)^m \times X_\tau$, and the restriction of $(\mu_{\mathrm{id}_P, \eta})_*\tsfC_{\mathrm{id}_P}^{P}((\scA_i)_{i \in P_*})$ identifies with the external product of $\scA_*$ with some convolution products of the $\scA_i$'s and with $\underline{\bk}_{X_\tau}[|P|]$. Using the fact that convolution of $L^+G$-equivariant perverse sheaves on $\Gr_G$ is t-exact (see~\S\ref{ss:Satake-category})
we see that if $m<|P|$ this restriction is in negative perverse degrees, proving the desired claim about restriction. The claim about corestrictions can be checked similarly, or deduced using Verdier duality.  This completes the proof of~\eqref{eqn:conv-calc-special}.

To prove the lemma in general, we use the commutative diagram in Figure~\ref{fig:conv-calc}.  Our problem lies along the diagonal of this diagram.  Across the top of the diagram is an instance of~\eqref{eqn:convdiag-alpha}, and down the right-hand side of the diagram is an instance of~\eqref{eqn:convdiag-mu}.  The squares involving maps labeled ``$m$'' are all cartesian.  We have
\[
(\mu_{\alpha,\eta})_*\tsfC_\alpha^{P}((\scA_i)_{i \in P_*})
\cong (\mu_{\mathrm{id}_Q,\eta})_* ((m_\alpha^{P,\varnothing})_\eta)_*\tsfC_\alpha^{P}((\scA_i)_{i \in P_*}).
\]
By proper base change, we have
\[
p_\eta^* ((m_\alpha^{P,\varnothing})_\eta)_*\tsfC_\alpha^{P}((\scA_i)_{i \in P_*})
\cong 
q_\eta^* ((m_\alpha^{\nmin(\alpha^{-1}),\varnothing})_\eta)_*\tsfC_\alpha^{\nmin(\alpha^{-1})}((\scA_i)_{i \in P_*}),
\]
and then by Lemma~\ref{lem:conv-calc1} we have
\[
((m_\alpha^{P,\varnothing})_\eta)_*\tsfC_\alpha^{P}((\scA_i)_{i \in P_*})
\cong \tsfC_{\mathrm{id}_Q}^{Q}((\scB_j)_{j \in Q_*}).
\]
Now apply $(\mu_{\mathrm{id}_Q,\eta})_*$ to this equation.  The result follows by the special case~\eqref{eqn:conv-calc-special} considered above.
\end{proof}

\begin{figure}
\[
\begin{tikzcd}
\ctGr_\alpha^{\varnothing,\varnothing} &
\ctGr_\alpha^{\varnothing,\nmin(\alpha^{-1})} \ar[l, "q"'] \ar[r, "p"] &
\ctGr_\alpha^{\nmin(\alpha^{-1}),\varnothing} \ar[r, "m"] & \ctGr_{Q}^{\varnothing, \varnothing} \\
& \ctGr_\alpha^{\varnothing, P} \ar[u, "q"] \ar[ul, "q"] \ar[r, "p"] \ar[dr, "p"'] 
& \ctGr_\alpha^{\nmin(\alpha^{-1}), \min(\alpha^{-1})} \ar[u, "q"] \ar[r, "m"] \ar[d, "p"] 
& \ctGr_{Q}^{\varnothing, Q} \ar[u, "q"'] \ar[d, "p"] \\
&& \ctGr_\alpha^{P, \varnothing} \ar[r, "m"] \ar[dr, "\mu"']
& \ctGr_{Q}^{Q, \varnothing} \ar[d, "\mu"] \\
&&& \bGr_Q
\end{tikzcd}
\]
\caption{Diagram for the proof of Lemma~\ref{lem:conv-calc2}}\label{fig:conv-calc}
\end{figure}

\subsection{Proof of Theorem~\ref{thm:nc-main-new}}
\label{ss:proof-new}

We will first establish the existence of and formula for $\Upsilon^\alpha_{\bGr_P}(\sfC_P((\scA_i)_{i \in P_*}))$.  Choose total orders on $P_*$ and $Q_*$ as~\S\ref{ss:conv-torsor}, so that $*$ is the smallest element in both sets, and such that $\alpha: P_* \to Q_*$ is order-preserving.  Consider the diagram
\[
\ctGr_P^{\varnothing,\varnothing} \xleftarrow{q}
\ctGr_P^{\varnothing,P} \xrightarrow{p} \ctGr_P^{P,\varnothing}
\xrightarrow{\mu_{\mathrm{id}_P}} \bGr_P.
\]
Its base change along $\bar\alpha: \bA^Q \to \bA^P$ is
\[
\ctGr_\alpha^{\varnothing,\varnothing} \xleftarrow{q}
\ctGr_\alpha^{\varnothing,P} \xrightarrow{p} \ctGr_\alpha^{P,\varnothing}
\xrightarrow{\mu_\alpha} \bGr_Q.
\]
To start, in view of~\eqref{eqn:ctGr-product}, Lemma~\ref{lem:product-upexist}, and Remark~\ref{rmk:thm-special}\eqref{it:thm-special-1},
\[
\Upsilon^\alpha_{\ctGr_P^{\varnothing,\varnothing}}\left(\bigboxtimes_{i \in P_*} \scA_i  \boxtimes \underline{\bk}_{(C \smallsetminus \{0\})^P}[|P|]\right)
\]
is well defined, and isomorphic to 
\[
\left( \bigboxtimes_{i \in P_*} \scA'_i \right) \boxtimes \underline{\bk}_{(C \smallsetminus \{0\})^Q}[|Q|]
\quad\text{where}\quad
\scA'_i =
\begin{cases}
\sfZ(\scA_i) & \text{if $i \in P \cap \alpha^{-1}(*)$,} \\
\scA_i & \text{otherwise.}
\end{cases}
\]
Next, by two applications of Lemma~\ref{lem:compat-smooth}, we obtain that
\[
\Upsilon^\alpha_{\ctGr_P^{P,\varnothing}}(\tsfC_{\mathrm{id}_P}^{P}((\scA_i)_{i \in P_*}))
\]
is well defined, and isomorphic to
\begin{equation}\label{eqn:sfc-mu-calc}
\tsfC_\alpha^P((\scA'_i)_{i \in P_*}).
\end{equation}
Applying Lemma~\ref{lem:conv-calc2} twice, we have
\[
(\mu_{\mathrm{id}_P,\eta})_*\tsfC_{\mathrm{id}_P}^{P}((\scA_i)_{i \in P_*}) = \sfC_P((\scA_i)_{i \in P_*}), \ \
(\mu_{\alpha,\eta})_*\tsfC_\alpha^P((\scA'_i)_{i \in P_*})
= \sfC_Q((\scB_i)_{i \in Q_*}),
\]
where the $\scB_j$'s are as in Theorem~\ref{thm:nc-main-new}.
By Lemma~\ref{lem:compat-proper}, we conclude that the $\alpha$-nearby cycles of $\sfC_P((\scA_i)_{i \in P_*})$ are well defined, and that
\[
\Upsilon^\alpha_{\bGr_P}(\sfC_P((\scA_i)_{i \in P_*})) \cong \sfC_Q((\scB_i)_{i \in Q_*}),
\]
as desired.  This completes the proof of the first part of Theorem~\ref{thm:nc-main-new}.

Next, by Lemma~\ref{lem:composition-product}, the natural map
\[
\Upsilon^{\beta\alpha}_{\ctGr_P^{\varnothing,\varnothing}}\left(\bigboxtimes_{i \in P_*} \scA_i  \boxtimes \underline{\bk}_{(C \smallsetminus \{0\})^P}[|P|]\right)
\to
\Upsilon^\beta_{\ctGr_\alpha^{\varnothing,\varnothing}}\Upsilon^\alpha_{\ctGr_P^{\varnothing,\varnothing}}\left(\bigboxtimes_{i \in P_*} \scA_i  \boxtimes \underline{\bk}_{(C \smallsetminus \{0\})^P}[|P|]\right)
\]
is an isomorphism.  By two applications of Lemma~\ref{lem:compat-compose-smooth}, we find that the map
\[
\Upsilon^{\beta\alpha}_{\ctGr_P^{P,\varnothing}}(\tsfC_P^{P}((\scA_i)_{i \in P_*})) \to
\Upsilon^\beta_{\ctGr_\alpha^{P,\varnothing}}\Upsilon^\alpha_{\ctGr_P^{P,\varnothing}}(\tsfC_P^{P}((\scA_i)_{i \in P_*}))
\]
is an isomorphism, and then Lemma~\ref{lem:compat-compose-proper} implies that so is the map
\[
\Upsilon^{\beta\alpha}_{\bGr_P}(\sfC_P((\scA_i)_{i \in P_*}) \to \Upsilon^\beta_{\bGr_Q}(\Upsilon^\alpha_{\bGr_P}(\sfC_P((\scA_i)_{i \in P_*}))
\]
This completes the proof of Theorem~\ref{thm:nc-main-new}.

\subsection{Groupoid perspective}
\label{ss:groupoid}

Let $P$ be a finite set.  Let $K = K_P$ be the set whose elements are sequences of surjective pointed maps
\begin{equation}\label{eqn:gamma-ex}
\gamma = (P_* \xrightarrow{\alpha_1} P_{1*} \xrightarrow{\alpha_2} \cdots \xrightarrow{\alpha_{k-1}} P_{k-1,*} \xrightarrow{\alpha_k} \varnothing_*).
\end{equation}
Given such a sequence $\gamma$, an \emph{elementary refinement} of $\gamma$ is a new sequence $\gamma'$ obtained by decomposing some $\alpha_i$ into a composition of two surjective maps: say
\[
\gamma' = (P_* \xrightarrow{\alpha_1} \cdots \xrightarrow{\alpha_{i-1}} P_{i-1,*} \xrightarrow{\alpha'_i} Q_* \xrightarrow {\alpha''_i} P_{i,*} \xrightarrow{\alpha_{i+1}} \cdots \xrightarrow{\alpha_k} \varnothing_*)
\]
where $\alpha_i = \alpha''_i \circ \alpha_i'$.  Make $K$ into a poset by declaring that $\gamma \preceq \gamma'$ if $\gamma'$ can be obtained from $\gamma$ by a (possibly empty) sequence of elementary refinements.  Of course, this poset can be regarded as a category in the usual way: there is a morphism $\gamma \to \gamma'$ if $\gamma \preceq \gamma'$.  This poset (resp.~category) has a unique minimal element (resp.~initial object): namely, the unique pointed map $P_* \to \varnothing_*$.

\begin{lem}\label{lem:nerve-contract}
Let $K^{\cong}$ be the groupoid obtained from $K$ by formally inverting all morphisms.  Then $K^{\cong}$ is a contractible groupoid.  
\end{lem}

Recall that a groupoid is said to be \emph{contractible} if for any two objects $x$ and $y$, there is a unique morphism $x \to y$.  (This is equivalent to requiring that the nerve of the groupoid be a contractible Kan complex.)  The following standard argument applies to any poset with a unique minimal (or maximal) element.

\begin{proof}
The initial object $e$ of $K$ remains an initial object in $K^{\cong}$, so there is a unique morphism from $e$ to every other object.  This implies that every object of $K^{\cong}$ is initial, and then that $K^{\cong}$ is contractible.
\end{proof}

Let $\gamma$ be as in~\eqref{eqn:gamma-ex}.  For brevity, we introduce the notation
\[
\underline{\Upsilon}^\gamma(\sfC_P((\scA_i)_{i \in P_*})) :=  
\Upsilon_{\bGr_{P_{k-1}}}^{\alpha_k} \circ \cdots \circ \Upsilon_{\bGr_{P_1}}^{\alpha_2} \circ \Upsilon_{\bGr_P}^{\alpha_1}(\sfC_P((\scA_i)_{i \in P_*})).
\]
(Here, all the functors are well defined thanks to Theorem~\ref{thm:nc-main-new}.)

\begin{prop}
\label{prop:bez-fink}
For any object $\scA_*$ in $\Perv_I(\Fl_G,\bk)$, and any collection of objects $(\scA_p)_{p \in P}$ in $\Perv_{L^+ G}(\Gr_G,\bk)$, there is a contractible groupoid whose objects are of the form $\underline{\Upsilon}^\gamma(\sfC_P((\scA_i)_{i \in P_*}))$.
\end{prop}

\begin{proof}
Define a functor $F: K \to \Perv_I(\Fl_G,\bk)$ as follows: on objects, we set 
\[
F(\gamma) = \underline{\Upsilon}^\gamma(\sfC_P((\scA_i)_{i \in P_*})).
\]
If $\gamma \to \gamma'$ is an elementary refinement, then Theorem~\ref{thm:nc-main-new} gives us an isomorphism
\[
F(\gamma \to \gamma'): F(\gamma) \simto F(\gamma').
\]
By Lemma~\ref{lem:commnc-assoc}, this rule extends to arbitrary morphisms in $K$, so $F$ is a well-defined functor.  Since $F$ sends every morphism in $K$ to an isomorphism, it extends uniquely to a faithful functor $F^{\cong}: K^{\cong} \to \Perv_I(\Fl_G,\bk)$.  Its image is a (non-full) subcategory of $\Perv_I(\Fl_G,\bk)$ that is a contractible groupoid by Lemma~\ref{lem:nerve-contract}.
\end{proof}

\begin{rmk}
Here are some examples of objects in the groupoid from Proposition~\ref{prop:bez-fink} in the case where $P = \{1,\ldots,n\}$ (cf.~Remark~\ref{rmk:thm-special}).  Choose an enumeration $\{\chi_1,\ldots,\chi_n\}$ of $\{1,\ldots,n\}$, and let $\gamma_\chi$ be the sequence
\[
P_* = \{\chi_1,\ldots,\chi_n\}_* \xrightarrow{\alpha_1} \{\chi_2,\ldots,\chi_n\}_* \to 
\cdots \to \{\chi_{n-1},\chi_n\}_* \xrightarrow{\alpha_{n-1}} \{\chi_n\}_* \xrightarrow{\alpha_n} \varnothing_*,
\]
where $\alpha_i(\chi_i) = *$ and $\alpha_i(\chi_j) = \chi_j$ for $j > i$.  Let $f_{\chi_i}$ denote the composition 
\[
\bGr_{\{1, \ldots, n\}} \times_{\bA^{\{1,\ldots,n\}}} \bA^{\{\chi_i, \chi_{i+1},\ldots,\chi_n\}} \to \bA^{\{\chi_i, \chi_{i+1},\ldots,\chi_n\}} \to \bA^{\{\chi_i\}}.
\]
By Example~\ref{ex:Q-empty} and Lemma~\ref{lem:classical-nc}, we have
\[
\underline{\Upsilon}^{\gamma_\chi}((\scA_i)_{i \in P_*}) \cong
\Psi_{f_{\chi_n}} \cdots \Psi_{f_{\chi_1}}(\sfC_P((\scA_i)_{i \in P_*})).
\]
On the other hand, if we let $\gamma_{\min}$ denote the unique map $P_* \to \varnothing_*$, then Theorem~\ref{thm:nc-main-new} says that
\[
\underline{\Upsilon}^{\gamma_{\min}}((\scA_i)_{i \in P_*}) \cong \scA_* \star^I \left(\bigstar_{p \in P}^I \sfZ(\scA_p) \right).
\]
Our considerations therefore fully justify~\cite[Proposition~5.2.1]{bbfk}.
\end{rmk}


\begin{thebibliography}{99}

\bibitem{achar-book}
P.~Achar, \emph{Perverse sheaves and applications to representation theory},
Mathematical Surveys and Monographs 258, American Mathematical Society, Providence, RI, 2021.

\bibitem{arbook}
P.~Achar and S.~Riche, {\em Central sheaves on affine flag varieties},
  approx.~350 pp., draft available at \url{https://lmbp.uca.fr/~riche/}.
  
  \bibitem{br}
P.~Baumann and S.~Riche, \emph{Notes on the geometric Satake equivalence}, 
in \emph{Relative Aspects in Representation Theory, Langlands Functoriality and Automorphic Forms, CIRM Jean-Morlet Chair, Spring 2016} (V. Heiermann, D. Prasad, Eds.), 1--134, Lecture Notes in Math. 2221, Springer, 2018.
  
  \bibitem{beilinson}
A. Be{\u\i}linson, \emph{How  to  glue  perverse  sheaves}, in \emph{{$K$}-theory, arithmetic and geometry
({M}oscow, 1984--1986)}, 42--51, Lecture Notes in Math. 1289, Springer-Verlag, 1987.
  
\bibitem{bbd}
A.~Be{\u\i}linson, J.~Bernstein, P.~Deligne, and O. Gabber, \emph{Faisceaux pervers}, in \emph{Analyse et topologie sur les espaces singuliers, I (Luminy, 1981)}, 5--172,
Ast\'erisque \textbf{100}, Soc. Math. France, 1982.
  
\bibitem{bez}
R.~Bezrukavnikov, \emph{On tensor categories attached to cells in affine Weyl groups}, 
in \emph{Representation theory of algebraic groups and quantum groups}, 69--90,
Adv. Stud. Pure Math. 40, Math. Soc. Japan, Tokyo, 
2004. 

\bibitem{bbfk}
R.~Bezrukavnikov, A.~Braverman, M.~Finkelberg, and D.~Kazhdan, \emph{Schwartz spaces, local L-factors and perverse sheaves}, preprint~\href{https://arxiv.org/abs/2303.00913}{arXiv:2303.00913}.

\bibitem{gaitsgory}
D.~Gaitsgory,
\emph{Construction of central elements in the affine Hecke algebra via nearby cycles},
Invent. Math.~\textbf{144} (2001), no. 2, 
253--280.

\bibitem{gaitsgory-app}
D.~Gaitsgory, \emph{Appendix: braiding compatibilities}, in \emph{Representation theory of algebraic groups and quantum groups}, 91--100, Adv. Stud. Pure Math. 40, Math. Soc. Japan, Tokyo, 2004. 

\bibitem{illusie-autour}
L.~Illusie, \emph{Autour du th\'eor\`eme de monodromie locale}, in
\emph{P\'eriodes $p$-adiques (Bures-sur-Yvette, 1988)},
Ast\'erisque 
\textbf{223} (1994), 9--57.

\bibitem{illusie}
L.~Illusie, \emph{Vanishing cycles over general bases, after P. Deligne, O. Gabber, G. Laumon and F. Orgogozo}, available at~\url{https://www.imo.universite-paris-saclay.fr/~luc.illusie/vanishing1b.pdf}.

\bibitem{mv}
I.~Mirkovi{\'c} and K.~Vilonen, \emph{Geometric Langlands duality and representations of algebraic groups over commutative rings}, Ann. of Math. (2) \textbf{166} (2007), no. 1, 95--143. 

\bibitem{morel}
S.~Morel, \emph{Beilinson's construction of nearby cycles and gluing}, notes available at~\url{http://perso.ens-lyon.fr/sophie.morel/gluing.pdf}, 2018.

\bibitem{reich}
R.~Reich, \emph{Notes on Beilinson's ``How to glue perverse sheaves''},
J. Singul. \textbf{1} (2010), 94--115.

\bibitem{salmon}
A.~Salmon, \emph{Unipotent nearby cycles and the cohomology of shtukas}, 
Compos. Math. \textbf{159} (2023), no. 3, 590--615.

\bibitem{salmon2}
A.~Salmon, {\em Unipotent nearby cycles and nearby cycles over general bases},
  preprint arXiv:2401.16746.

\end{thebibliography}
\end{document}